\documentclass[10pt]{amsart}

\usepackage{times,amsmath,amsbsy,amssymb,amscd,mathrsfs}
\usepackage{graphicx,subfigure,epstopdf,wrapfig,chemarrow}

\usepackage{algorithm2e} 
\usepackage{multicol,multirow}
\usepackage{mathtools}
\usepackage[usenames,dvipsnames,svgnames,table]{xcolor}
\usepackage[all]{xy}
\usepackage{wrapfig}
\usepackage{tcolorbox}

\usepackage{tikz,tikz-cd}
\usepackage[utf8]{inputenc}
\usepackage{pgfplots} 
\usepackage{pgfgantt}
\usepackage{pdflscape}
\pgfplotsset{compat=newest} 
\pgfplotsset{plot coordinates/math parser=false}
\newlength\fwidth

\definecolor{myBlue}{rgb}{0.0,0.0,0.55}
\usepackage[pdftex,colorlinks=true,citecolor=myBlue,linkcolor=myBlue]{hyperref}

\usepackage[hyperpageref]{backref}

\usepackage{comment,enumerate,multicol,xspace}

  \newcounter{mnote}
  \let\oldmarginpar\marginpar
    \renewcommand\marginpar[1]{\-\oldmarginpar[\raggedleft\footnotesize #1]%
    {\raggedright\footnotesize #1}}

\newtheorem{theorem}{Theorem}[section]
\newtheorem{lemma}[theorem]{Lemma}
\newtheorem{corollary}[theorem]{Corollary}

\newtheorem{remark}[theorem]{Remark}

\newcommand{\dx}{\,{\rm d}x}
\newcommand{\dd}{\,{\rm d}}

\newcommand{\bs}{\boldsymbol}

\newcommand{\curl}{{\rm curl\,}}
\renewcommand{\div}{\operatorname{div}}
\newcommand{\grad}{{\rm grad\,}}
\DeclareMathOperator*{\rot}{rot}

\newcommand{\tr}{\operatorname{tr}}
\newcommand{\dev}{\operatorname{dev}}

\newcommand{\mskw}{\operatorname{mskw}}

\newcommand{\step}[1]{\noindent\raisebox{1.5pt}[10pt][0pt]{\tiny\framebox{$#1$}}\xspace}

\newcommand{\vertiii}[1]{{\left\vert\kern-0.25ex\left\vert\kern-0.25ex\left\vert #1 
    \right\vert\kern-0.25ex\right\vert\kern-0.25ex\right\vert}}

\usepackage{booktabs}
\usepackage{stmaryrd}
\usepackage{scalefnt}
\usepackage{graphicx}
\allowdisplaybreaks[1]
\newcommand{\Oplus}{\ensuremath{\vcenter{\hbox{\scalebox{1.5}{$\oplus$}}}}}

\begin{document}
\title[Finite Element curl\,div Complex]{Distributional Finite Element curl\,div Complexes and Application to Quad Curl Problems}
\author{Long Chen}%
 \address{Department of Mathematics, University of California at Irvine, Irvine, CA 92697, USA}%
 \email{chenlong@math.uci.edu}%
 \author{Xuehai Huang}%
 \address{School of Mathematics, Shanghai University of Finance and Economics, Shanghai 200433, China}%
 \email{huang.xuehai@sufe.edu.cn}%
\author{Chao Zhang}%
\address{School of General Education, Wenzhou Business College, Wenzhou 325035, China}%
\email{zcmath@163.sufe.edu.cn}%

\thanks{The first author was partially  supported by NSF DMS-2012465, DMS-2309777, and DMS-2309785. The second author was supported by the National Natural Science Foundation of China Project 12171300, and the Natural Science Foundation of Shanghai 21ZR1480500.}

\makeatletter
\@namedef{subjclassname@2020}{\textup{2020} Mathematics Subject Classification}
\makeatother

 \begin{abstract}
The paper addresses the challenge of constructing finite element $\curl\div$ complexes in three dimensions. Tangential-normal continuity is introduced in order to develop distributional finite element $\curl\div$ complexes. The spaces constructed are applied to discretize the quad curl problem, demonstrating optimal order of convergence. Furthermore, a hybridization technique is proposed, demonstrating its equivalence to nonconforming finite elements and weak Galerkin methods.
\end{abstract}


\maketitle


\section{Introduction}
In this work, we will 
construct 
distributional finite element $\text{curl}\,\text{div}$ complexes in three dimensions, and apply it to solve 
the fourth-order curl problem $-\curl\Delta\curl \boldsymbol{u}=\boldsymbol{f}, \div \bs u = 0$ in a domain $\Omega\subset \mathbb R^3$ with boundary conditions $\boldsymbol{u} \times \boldsymbol{n}=\operatorname{curl} \boldsymbol{u} =0$ on $\partial \Omega$. Such problem arises from multiphysics simulation such as modeling a
magnetized plasma in magnetohydrodynamics~\cite{kingsep1990reviews,Chacon;Simakov;Zocco:2007Steady-state}.

We first give a brief literature review on distributional finite elements. The distributional finite element de Rham complexes are adopted to construct equilibrated residual error estimators in~\cite{BraessSchoeberl2008}, which are then extended to discrete distributional differential forms in~\cite{Licht2017}, discrete distributional elasticity complexes in~\cite{Christiansen2011}, and discrete distributional Hessian and $\text{divdiv}$ complexes in~\cite{HuLinZhang2023} with applications in cohomology groups. Recently, in~\cite{ChenHuang2023}, the distributional finite element $\text{divdiv}$ element has been constructed and applied for solving the mixed formulation of the biharmonic equation in arbitrary dimensions. The distributional finite elements allow the use of piecewise polynomials with less smoothness, which is especially useful for high-order differential operators.

Let us use a more familiar $2$nd order operator $\nabla^2$ as an example to illustrate the motivation. The $H^2$-conforming finite element on tetrahedron meshes~\cite{HuLinWu2023,ChenHuang2021,ChenHuang2024,Zhang:2009family} requires polynomials of degree $9$ and above and possesses extra smoothness at vertices and edges. Therefore, it is hardly used in practice. Simple finite elements can be constructed if the differential operators are understood in the distribution sense.

For the discretization of the biharmonic equation in two dimensions, the so-called Hellan-Herrmann-Johnson (HHJ) mixed method~\cite{Hellan1967,Herrmann1967,Johnson1973} requires only normal-normal continuous finite elements for symmetric tensors and thus $C^0$-conforming Lagrange element, not $C^1$-conforming elements,  can be used for displacement. This normal-normal continuous finite element is then employed to solve linear elasticity~\cite{PechsteinSchoeberl2011} and Reissner-Mindlin plates~\cite{PechsteinSchoeberl2017}, and used to construct the first two-dimensional distributional finite element $\text{divdiv}$ complexes in~\cite{ChenHuHuang2018a}. Recently, the distributional finite element $\text{divdiv}$ element for solving the mixed formulation of the biharmonic equation has been extended to arbitrary dimensions in~\cite{ChenHuang2023}.


Now we move to the curl div operator. Introduce the space 
$H(\curl\div,\Omega; \mathbb T) := \left\{\boldsymbol\tau \in L^2(\Omega; \mathbb T): \curl \div \boldsymbol\tau\in L^{2}(\Omega;\mathbb R^3)\right\}$, where $\mathbb T$ is the space of traceless tensors.
A mixed formulation of the quad-curl problem is to find $\boldsymbol{\sigma}\in H(\operatorname{curl} \div, \Omega ; \mathbb{T})$, $\boldsymbol{u}\in L^2(\Omega;\mathbb R^3)$ and $\phi\in H_0^{1}(\Omega)$ such that
\begin{align*}
(\boldsymbol{\sigma},\boldsymbol{\tau})+b(\boldsymbol{\tau},\psi;\boldsymbol{u})&=0,  & & \boldsymbol{\tau} \in H(\operatorname{curl} \div, \Omega ; \mathbb{T}), \psi\in H_0^{1}(\Omega), \\
b(\boldsymbol{\sigma},\phi;\boldsymbol{v}) &=-\langle \boldsymbol{f}, \boldsymbol{v}\rangle, & &\boldsymbol{v} \in L^2(\Omega;\mathbb R^3),
\end{align*}
where the bilinear form 
$b(\boldsymbol{\tau},\psi;\boldsymbol{v}):=(\operatorname{curl} {\div}\boldsymbol{\tau}, \boldsymbol{v})+(\grad\psi, \boldsymbol{v}).$
The term $(\grad\psi, \boldsymbol{u})$ is introduced to impose the divergence free condition $\div \bs u = 0$. 

Finite element spaces conforming to $H(\curl\div,\Omega; \mathbb T)$
are relatively complicated due to the smoothness  requirement $\curl \div \boldsymbol\tau\in L^{2}(\Omega;\mathbb R^3)$. In the distributional sense
\begin{equation*}
\langle\curl\div \boldsymbol{\tau}, \boldsymbol{v}\rangle =-(\boldsymbol{\tau}, \grad\curl\boldsymbol{v}), \quad \bs v\in H(\grad\curl,\Omega),
\end{equation*}
the smoothness can be shifted to the test function $\bs v$, where $H(\grad\curl,\Omega) := \{\bs u \in  L^2(\Omega;\mathbb R^3): \curl \bs u \in H^1(\Omega;\mathbb R^3) \}$. Of course $H(\grad\curl)$-conforming finite elements are not easy to construct neither. For example, the $H(\grad\curl)$-conforming finite elements are constructed in~\cite{ZhangZhang2020,HuZhangZhang2020,ZhangWangZhang2019,ChenHuang2022,ChenHuang2024}, which requires polynomial of degree at least $7$ and dimension of shape function space at least $315$. 

The key idea is to strike a balance of the smoothness of the trial function $\bs \tau$ and the test function $\bs v$. Given a mesh $\mathcal T_h$, let  $H^s(\mathcal T_h)$ be the space of piecewise $H^s$ function. Introduce the traceless tensor space with tangential-normal continuity 
\begin{align*}
\Sigma^{\rm tn}:=\{\boldsymbol{\tau}  \in H^{1}(\mathcal T_h ; \mathbb{T}) :  \left.\llbracket\boldsymbol{n}\times\boldsymbol{\tau}\boldsymbol{n} \rrbracket\right|_{F}=0 \text { for each } F \in \mathring{\mathcal{F}}_{h}\},
\end{align*}
where $\llbracket\boldsymbol{n}\times\boldsymbol{\tau}\boldsymbol{n} \rrbracket$ is the jump of $\boldsymbol{n}\times\boldsymbol{\tau}\boldsymbol{n}$ across all interior faces $F$. While for the test space, we use space 
\begin{equation*}
V_0^{\curl} := H_0(\curl,\Omega)\cap H^1(\curl,\mathcal T_h), 
\end{equation*}
where $H^1(\curl,\mathcal T_h):=\{\boldsymbol{v}\in H^1(\mathcal T_h;\mathbb R^3): \curl\boldsymbol{v}\in H^1(\mathcal T_h;\mathbb R^3)\}$.
Define a weak operator
$
(\curl\div)_w: \Sigma^{\rm tn} \to (V_0^{\curl})'
$
by
\begin{equation}\label{intro:weakcurldiv}
\langle (\curl\div)_w \boldsymbol{\tau}, \boldsymbol{v}\rangle :=\sum_{T\in \mathcal{T}_{h}}(\operatorname{div} \boldsymbol{\tau}, \curl\boldsymbol{v})_{T}-\sum_{F\in \mathring{\mathcal{F}}_{h}}(\llbracket \boldsymbol{n}^{\intercal}\boldsymbol{\tau}\boldsymbol{n} \rrbracket, \boldsymbol{n}_F\cdot\curl\boldsymbol{v})_{F},
\end{equation}
which is analog to the weak divdiv operator in HHJ mixed method~\cite{Hellan1967,Herrmann1967,Johnson1973}.
Now the function $\bs \tau$ is tangential-normally continuous and $\bs v$ is tangentially continuous so that $\boldsymbol{n}_F\cdot\curl\boldsymbol{v} = \rot_F \bs v$ is continuous on face $F$. One can easily show $(\curl\div)_w\bs \tau = \curl \div \bs \tau$ in the distribution sense by taking $\bs v\in C_0^{\infty}(\Omega;\mathbb R^3)$ in \eqref{intro:weakcurldiv}. 


We will use the tangential-normal continuous finite element constructed in~\cite{GopalakrishnanLedererSchoeberl2020a} for the discretization of $\Sigma^{\rm tn}$. 
For an integer $k\geq 0$, take $\mathbb P_{k}(T;\mathbb T)$ as the space of shape functions. The degrees of freedom (DoFs) are given by:
\begin{subequations}
\begin{align}
\label{intro:curdivfemdof1} \int_{F} \boldsymbol{t}_i^{\intercal}\boldsymbol{\tau}\boldsymbol{n}\, q \dd S, \quad & q \in \mathbb P_{k}(F), i = 1,2, F\in\mathcal F(T), \\
\label{intro:curdivfemdof2}  \int_{T} \boldsymbol{\tau} : \boldsymbol{q} \dx, \quad & \boldsymbol{q} \in \mathbb{P}_{k-1}(T; \mathbb{T}),
\end{align}
\end{subequations}
where $\boldsymbol t_{1}$ and $\boldsymbol t_{2}$ denote two
mutually perpendicular unit tangential vectors of face $F$ and are used to determine the tangential component of the vector $\bs \tau\bs n$.
The global finite element space $\Sigma_{k,h}^{\rm tn}$ by requiring single valued \eqref{intro:curdivfemdof1} is tangential-normally continuous. 

We use N\'ed\'elec elements $\mathring{\mathbb V}_{h}^{\curl}\subset V_0^{\curl}$ for the tangential continuous vector space, and use the Riesz representation of the $L^2$-inner product to bring the abstract dual to a concrete function. Define
$(\curl\div)_h:  \Sigma_{h}^{\rm tn} \to \mathring{\mathbb V}_{h}^{\curl}$ such that
\begin{equation}\label{intro:curldivh}
((\curl\div)_h\bs{\tau}_h, \bs v_h) = \langle (\curl\div)_w \boldsymbol{\tau}_h, \boldsymbol{v}_h\rangle, \quad \bs v_h\in \mathring{\mathbb V}_{h}^{\curl},
\end{equation}
and its $L^2$-adjoint operator $(\grad\curl)_h: \mathring{\mathbb V}_{h}^{\curl}  \to \Sigma_{h}^{\rm tn}$. 

By including the tensor version of the N\'ed\'elec elements $\mathbb V_{k,h}^{\curl}(\mathbb{M})$, and the Lagrange elements $\mathbb V_{k+1,h}^{\grad}(\mathbb{R}^{3})$, we are able to construct the distributional finite element $\curl\div$ complex:
\begin{equation}\label{intro:curldivfecomplex}
\begin{aligned}
\mathbb{R}^{3} \times\{0\} \rightarrow  \mathbb V_{k+1,h}^{\grad}(\mathbb{R}^{3}) \times \mathbb{R} &\xrightarrow{(\grad\!, \,\mskw\boldsymbol{x})} \mathbb V_{k,h}^{\curl}(\mathbb{M})\xrightarrow{\dev\curl}\\ 
&\Sigma_{k-1,h}^{\rm tn}
 \xrightarrow{(\curl\div)_h} \mathring{\mathbb V}_{(k,\ell), h}^{\curl} \xrightarrow{\div_h} \mathring{\mathbb V}_{\ell+1, h}^{\grad} \rightarrow 0,
\end{aligned}
\end{equation}
where $(\grad, \mskw\boldsymbol{x})\begin{pmatrix}
\boldsymbol v\\
c
\end{pmatrix} = \grad\boldsymbol v + c\mskw\boldsymbol{x}$, and $\ell = k-1$ or $\ell = k$ is introduced to distinguish the first and second kind of N\'ed\'elec element. The lowest order, i.e. $k=1, \ell = 0$, of the last three elements are illustrated in Fig. \ref{fig:lowestorder}.

\begin{figure}[htbp]
\begin{center}
\includegraphics[width=8.5cm]{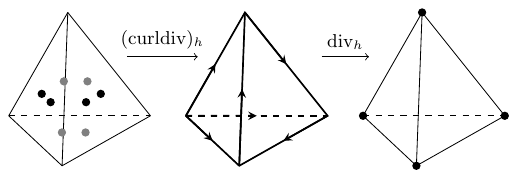}
\caption{The simplest elements $\Sigma_{0,h}^{\rm tn} - \mathring{\mathbb V}_{(1,0), h}^{\curl} - \mathring{\mathbb V}_{1, h}^{\grad}$: the first is a piecewise constant traceless matrix with tangential-normal continuity, the second is the lowest order edge element, and the third is the linear Lagrange element.}
\label{fig:lowestorder}
\end{center}
\end{figure}


The finite element complex \eqref{intro:curldivfecomplex} is a discretization of the distributional $\curl\div$ complex:
\begin{equation*}
\begin{aligned}
\mathbb{R}^{3} \times\{0\} \rightarrow &H^{1}(\Omega ; \mathbb{R}^{3}) \times \mathbb{R} \xrightarrow{(\grad, \mskw\boldsymbol{x})} H(\operatorname{curl}, \Omega ; \mathbb{M})\xrightarrow{\mathrm { dev~curl }}\\ 
&H^{-1}(\curl\div, \Omega ; \mathbb{T})
 \xrightarrow{\curl\div} H^{-1}(\operatorname{div}, \Omega) \xrightarrow{\div} H^{-1}(\Omega) \rightarrow 0,
\end{aligned}
\end{equation*}
where $H^{-1}(\Omega) := (H_0^1(\Omega))'$, and Sobolev spaces of negative order
\begin{equation*}
\begin{aligned}
H^{-1}(\div,\Omega) &:= \left\{\boldsymbol u \in H^{-1}(\Omega;\mathbb R^3): \div \boldsymbol u\in H^{-1}(\Omega)\right\},\\
H^{-1}(\curl\div,\Omega; \mathbb T) &:= \left\{\boldsymbol\tau \in L^2(\Omega; \mathbb T): \curl \div \boldsymbol\tau\in H^{-1}(\div,\Omega)\right\}.
\end{aligned}
\end{equation*}

As an application, we consider the fourth-order curl problem $-\curl\Delta\curl \boldsymbol{u}=\boldsymbol{f}, \div \bs u = 0$ with boundary conditions $\boldsymbol{u} \times \boldsymbol{n}=\operatorname{curl} \boldsymbol{u} =0$ on $\partial \Omega$. 
The distributional mixed finite element method is to find $\boldsymbol{\sigma}_h\in {\Sigma}^{\rm tn}_{k-1,h}$, $\boldsymbol{u}_{h}\in\mathring{\mathbb V}_{(k,\ell), h}^{\curl}$ and  $\phi_h\in\mathring{\mathbb V}_{\ell+1, h}^{\grad}$ such that
\begin{subequations}
\begin{align}
\label{intro:distribumfem1}	(\boldsymbol{\sigma}_h,\boldsymbol{\tau}_h)+b_h(\boldsymbol{\tau}_h,\psi_h;\boldsymbol{u}_h)&=0,  & & \boldsymbol{\tau}_h \in {\Sigma}^{\rm tn}_{k-1,h}, \psi_h\in\mathring{\mathbb V}_{\ell + 1, h}^{\grad}, \\
\label{intro:distribumfem2}	b_h(\boldsymbol{\sigma}_h,\phi_h;\boldsymbol{v}_h) &=-(\boldsymbol{f}, \boldsymbol{v}_h), & & \boldsymbol{v}_h \in \mathring{\mathbb V}_{(k,\ell), h}^{\curl},
\end{align}
\end{subequations}
where $b_h(\boldsymbol{\tau},\psi;\boldsymbol{v}):=((\curl\div)_h\boldsymbol{\tau}, \boldsymbol{v}) +(\grad\psi, \boldsymbol{v})$ and $(\curl\div)_h$ is a discretization of distributional $\curl\div$ operator, cf. \eqref{intro:curldivh}.

We prove two discrete inf-sup conditions and thus obtain the well-posedness of \eqref{intro:distribumfem1}-\eqref{intro:distribumfem2} and optimal order convergence
\begin{align*}
\|\boldsymbol{\sigma}-\boldsymbol{\sigma}_h\|_{0,h}+\|I_h^{\curl}\boldsymbol{u}-\boldsymbol{u}_h\|_{H((\grad\curl)_h)}&\lesssim h^k(|\boldsymbol{\sigma}|_k+|\boldsymbol{u}|_k), \\
\|\boldsymbol{u}-\boldsymbol{u}_h\|_{H(\curl)}+h\|\boldsymbol{u}-\boldsymbol{u}_h\|_{H((\grad\curl)_h)}&\lesssim h^k(|\boldsymbol{\sigma}|_k+|\boldsymbol{u}|_k+|\curl\boldsymbol{u}|_k).
\end{align*}
By the duality argument,
the order of $\|\curl(I_h^{\curl}\boldsymbol{u}-\boldsymbol{u}_h)\|$ can be improved to $h^{k+1}$ on convex domains. Both $\|I_h^{\curl}\boldsymbol{u}-\boldsymbol{u}_h\|_{H((\grad\curl)_h)}$ and $\|\curl(I_h^{\curl}\boldsymbol{u}-\boldsymbol{u}_h)\|$ are superconvergent.
Post-processing can be applied to improve the approximation to $\boldsymbol{u}$.

Furthermore, we apply hybridization techniques to \eqref{intro:distribumfem1}-\eqref{intro:distribumfem2}, leading to a stabilization-free weak Galerkin method and extending to the $H(\grad\curl)$ nonconforming finite elements introduced in~\cite{Huang2020,ZhengHuXu2011} for solving the quad-curl problem. Equivalently, we identify the complex that accommodates these nonconforming finite elements and generalize them to arbitrary orders.

For other discretization of the quad-curl problem, we refer to the macro finite element method in~\cite{HuZhangZhang2022}, nonconforming finite element methods in~\cite{HuangZhang2024,ZhangZhang2023,ZhangZhangZhang2023}, mixed finite element methods in~\cite{Sun2016,ChenLiQiuWang2021},
decoupled finite element methods in~\cite{CaoChenHuang2022,Zhang2018a,BrennerSunSung2017}, and references cited therein.

%

The rest of this paper is organized as follows. Section~\ref{sec:distribcurldivcomplex} focuses on the distributional $\curl\div$ complex.
A distributional finite element $\curl\div$ complex is constructed in Section~\ref{sec:distribfecurldivcomplex}, and applied to solve the quad-curl problem in Section~\ref{sec:mfem}.
The hybridization of the distributional mixed finite element method and the equivalence to other methods are presented in Section~\ref{sec:hybridization}. 

\section{Distributional $\curl\div$ complex}\label{sec:distribcurldivcomplex}
In this section, we present the distributional $\curl\div$ complex and introduce the weak differential operator $(\curl\div)_w$ which can be defined on the tangential-normal continuous matrix functions. 

\subsection{Notation}
Let $K\subset\mathbb R^3$ be a non-degenerated $3$-dimensional polyhedron. Denote by $\mathcal{F}(K)$ the set of all $2$-dimensional faces of $K$. For $F\in \mathcal{F}(K)$, denote by $\mathcal{E}(F)$ the set of all edges of $F$. Choose a normal vector $\bs n_F$ and two
mutually perpendicular unit tangential vectors $\boldsymbol t_{F,1}$ and $\boldsymbol t_{F,2}$, which will be abbreviated as $\boldsymbol t_{1}$ and $\boldsymbol t_{2}$ for simplicity.
Let $\boldsymbol{n}_K$ be the unit outward normal vector to $\partial K$, which will be abbreviated as $\boldsymbol{n}$. 
For $F\in \mathcal{F}(K)$ and $e\in\mathcal{E}(F)$, denote by $\boldsymbol{n}_{F,e}$ the unit vector being parallel to $F$ and outward normal to $\partial F$. 
Set $\boldsymbol{t}_{F,e}:=\boldsymbol{n}_{K}\times\boldsymbol{n}_{F,e}.$

Given a face $F\in \mathcal F(K)$, and a vector $\boldsymbol v\in \mathbb R^3$, define 
\begin{equation*}
\Pi_F\boldsymbol v= (\boldsymbol n\times \boldsymbol v)\times \boldsymbol n = (\boldsymbol I - \boldsymbol n\boldsymbol n^{\intercal})\boldsymbol v
\end{equation*}
as the projection of $\boldsymbol v$ onto the face $F$ which is called the tangential component of $\boldsymbol  v$. The vector $\boldsymbol  n\times \boldsymbol  v = (\bs n\times \Pi_F)\bs v$ is called the tangential trace of $\boldsymbol  v$, which is a rotation of $\Pi_F \boldsymbol  v$ on $F$ ($90^{\circ}$ counter-clockwise with respect to $\boldsymbol  n$).

Define the surface gradient operator as $\nabla_F: = \Pi_F \nabla$. For a scalar function \(v\), define the surface \(\curl\):
\begin{align*}
\curl_F v = \boldsymbol{n} \times \nabla v = \boldsymbol{n} \times \nabla_F v.
\end{align*}
For a vector function \(\boldsymbol{v}\), the surface rot operator is defined as:
\begin{equation*}
{\rm rot}_F \boldsymbol{v} :=  (\boldsymbol{n} \times \nabla) \cdot \boldsymbol{v} = (\boldsymbol{n} \times \nabla_F) \cdot \Pi_F\boldsymbol{v} = \boldsymbol{n} \cdot (\curl\boldsymbol{v}),
\end{equation*}
which represents the normal component of \(\curl\boldsymbol{v}\).


Denote the space of all $3\times 3$ matrices by $\mathbb{M}$, and all trace-free/traceless $3\times 3$ matrices by $\mathbb{T}$. Define the deviation $\dev\boldsymbol{\tau} = \boldsymbol{\tau}-\frac{1}{3}(\tr\boldsymbol{\tau})\boldsymbol I$. Obviously for a scalar function $u$, $\dev (u\bs I) = 0$.
For a vector $\boldsymbol{w}=(\omega_1, \omega_2, \omega_3)^{\intercal}\in\mathbb R^3$, let
\begin{equation*}\mskw\boldsymbol{w}:=\begin{pmatrix}
0 & -\omega_3 & \omega_2 \\
\omega_3 & 0 & -\omega_1 \\
-\omega_2 & \omega_1 & 0
\end{pmatrix}.
\end{equation*}
For a tensor-valued function $\boldsymbol{\tau}$, $\div\boldsymbol{\tau}$ and $\curl\boldsymbol{\tau}$ mean operators $\div$ and $\curl$ are applied row-wisely to $\boldsymbol{\tau}$. By direct calculation, we have the identities:
\begin{equation}\label{eq:divmskw}
\div \mskw \bs v = - \curl \bs v, \quad (\mskw \bs v)\bs n = \bs v \times \bs n.
\end{equation}

We use $\{\mathcal{T}_h\}_{h>0}$ to denote a shape regular family of simplicial meshes of $\Omega$ with mesh size $h=\max_{T\in\mathcal{T}_h}h_T$ and $h_T$ being the diameter of $T$.
Let $\mathcal{F}_{h}$, $\mathring{\mathcal{F}}_{h}$, $\mathcal{E}_{h}$, $\mathring{\mathcal{E}}_{h}$, $\mathcal{V}_{h}$ and $\mathring{\mathcal{V}}_{h}$ be the set of all faces, interior faces, edges, interior edges, vertices and interior vertices of $\mathcal T_h$, respectively. Let $T\in\mathcal T_h$ be a tetrahedron with four vertices $\texttt{v}_0, \ldots, \texttt{v}_3$.
Denote by $\lambda_i$ the $i$th barycentric coordinate with respect to the simplex $T$ for $i=0,\ldots, 3$. Set $\boldsymbol{t}_{ij}:=\texttt{v}_j-\texttt{v}_i$ as the edge vector from $\texttt{v}_i$ to $\texttt{v}_j$.

Given a nonnegative integer $k$, let $\mathbb P_k(T)$ stand for the set of all polynomials in $T$ with the total degree no more than $k$, and let $\mathbb P_k(T;\mathbb X)$ denote the tensor or vector version with $\mathbb X=\mathbb R^3$, $\mathbb M$ and $\mathbb T$. When $k<0$, set $\mathbb P_k(T) := \{0\}.$

Given a bounded domain $D\subset\mathbb{R}^{3}$ and a real number $s$, let $H^s(D)$ be the usual Sobolev space of functions
over $D$, whose norm and semi-norm are denoted by
$\Vert\cdot\Vert_{s,D}$ and $|\cdot|_{s,D}$ respectively. 
Let $(\cdot, \cdot)_D$ be the standard inner product on $L^2(D)$. If $D$ is $\Omega$, we abbreviate
$\Vert\cdot\Vert_{s,D}$, $|\cdot|_{s,D}$ and $(\cdot, \cdot)_D$ by $\Vert\cdot\Vert_{s}$, $|\cdot|_{s}$ and $(\cdot, \cdot)$, respectively. We also abbreviate $\Vert\cdot\Vert_{0,D}$ and $\Vert\cdot\Vert_{0}$ by $\Vert\cdot\Vert_{D}$ and $\Vert\cdot\Vert$, respectively.
The duality pair will be denoted by $\langle \cdot, \cdot \rangle$.

Introduce the following Sobolev spaces 
\begin{align*}
H(\curl,D) &:= \{\bs u \in  L^2(D;\mathbb R^3): \curl \bs u \in L^2(D;\mathbb R^3) \},\\
H(\div,D) &:= \{\bs u \in  L^2(D;\mathbb R^3): \div \bs u\in L^2(D) \}, \\
H(\grad\curl,D) &:= \{\bs u \in  L^2(D;\mathbb R^3): \curl \bs u \in H^1(D;\mathbb R^3) \},\\
H(\operatorname{curl} \div, D; \mathbb{T}) &:=\left\{\boldsymbol{\tau} \in L^{2}(D; \mathbb{T}): \operatorname{curl} \div\boldsymbol{\tau} \in L^2(D;\mathbb R^3)\right\},
\end{align*}
where $H^s(D;\mathbb X):=H^s(D)\otimes\mathbb X$.
Define piecewise smooth function space, for $s>0$, 
\begin{equation*}
H^s(\mathcal T_h) := \{ v\in L^2(\Omega): v\mid_{T}\in H^s(T) \text{ for all } T\in \mathcal T_h \},
\end{equation*}
and $H^s(\mathcal T_h;\mathbb X)$ its tensor or vector version with $\mathbb X=\mathbb R^3$, $\mathbb M$ and $\mathbb T$. Let $\grad_{\mathcal T_h}$ and $\curl_{\mathcal T_h}$ be the elementwise version of $\grad$ and $\curl$ associated with $\mathcal T_h$, respectively.

\subsection{The $\curl\div$ complexes}

The $\curl\div$ complex in three dimensions reads as~\cite[(47)]{ArnoldHu2021}
\begin{equation}\label{curldivcomplex}
\begin{aligned}
\mathbb{R}^{3} \times\{0\} \rightarrow &H^{1}(\Omega; \mathbb{R}^{3}) \times \mathbb{R} \xrightarrow{(\grad\!, \,\mskw\boldsymbol{x})} H(\operatorname{curl}, \Omega; \mathbb{M})\xrightarrow{\text { dev curl }}\\ 
&H(\curl\div, \Omega ; \mathbb{T})
 \xrightarrow{\curl\div} H(\operatorname{div}, \Omega) \xrightarrow{\div} L^{2}(\Omega) \rightarrow 0.
\end{aligned}
\end{equation}
When $\Omega$ is topologically trivial, i.e., all co-homology group of $\Omega$ is trivial, then  \eqref{curldivcomplex} is exact. The smoothness of the potential can be further improved to be in $H^1$. 

It is difficult to construct $H(\operatorname{curl} \div, \Omega ; \mathbb{T})$-conforming finite element with lower order degree of polynomials. To relax the smoothness, we are going to present a distributional $\curl\div$ complex with negative Sobolev spaces involved. 

Define 
\begin{equation*}
H^{-1}(\div,\Omega) = \{ \bs v\in H^{-1}(\Omega; \mathbb R^3): \div \bs v \in H^{-1}(\Omega)\}. 
\end{equation*}
In~\cite{ChenHuang2018}, we have shown that
$(H_{0}(\curl,\Omega))'=H^{-1}(\div,\Omega).$

Define 
\begin{equation*}
H^{-1}(\operatorname{curl} \div, \Omega ; \mathbb{T}):=\left\{\boldsymbol{\tau} \in L^{2}(\Omega ; \mathbb{T}): \operatorname{curl} \div\boldsymbol{\tau} \in H^{-1}(\div,\Omega)\right\}
\end{equation*}
with squared norm \begin{equation*}\|\boldsymbol{\tau}\|_{H(\operatorname{curl}\div)}^{2}:=\|\boldsymbol{\tau}\|^{2}+\|\curl\div \boldsymbol{\tau}\|_{H^{-1}(\div)}^{2}=\|\boldsymbol{\tau}\|^{2}+\|\curl\div \boldsymbol{\tau}\|_{-1}^{2}.\end{equation*}


\begin{lemma}
The distributional $\curl\div$ complex in three dimensions is 
\begin{equation}\label{distribcurldivcomplex}
\begin{aligned}
\mathbb{R}^{3} \times\{0\} \rightarrow &H^{1}(\Omega ; \mathbb{R}^{3}) \times \mathbb{R} \xrightarrow{(\grad, \mskw\boldsymbol{x})} H(\operatorname{curl}, \Omega ; \mathbb{M})\xrightarrow{\mathrm { dev~curl }}\\ 
&H^{-1}(\curl\div, \Omega ; \mathbb{T})
 \xrightarrow{\curl\div} H^{-1}(\operatorname{div}, \Omega) \xrightarrow{\div} H^{-1}(\Omega) \rightarrow 0.
\end{aligned}
\end{equation}
When $\Omega\subset\mathbb R^3$ is a bounded and topologically trivial Lipschitz domain, \eqref{distribcurldivcomplex} is exact.
\end{lemma}
\begin{proof}
Apparently \eqref{distribcurldivcomplex} is a complex. The surjection $\div H^{-1}(\operatorname{div}, \Omega)=H^{-1}(\Omega)$ follows from $\div L^{2}(\Omega;\mathbb R^3)=H^{-1}(\Omega)$ and $\div L^{2}(\Omega;\mathbb R^3)\subseteq \div H^{-1}(\operatorname{div}, \Omega)$. We then verify its exactness. 

\smallskip
\step 1 \; $\curl\div H^{-1}(\curl\div, \Omega ; \mathbb{T})= H^{-1}(\operatorname{div}, \Omega)\cap\ker(\div)$.

For $\boldsymbol{v}\in H^{-1}(\operatorname{div}, \Omega)\cap\ker(\div)$, by the exactness of the de Rham complex~\cite{CostabelMcIntosh2010}, there exists $\boldsymbol{\tau}\in H^1(\Omega;\mathbb M)$ such that $\boldsymbol{v}=\curl\div\boldsymbol{\tau}$. Notice that $\curl \div (p\bs I) = \curl \grad p = 0$. Then $\boldsymbol{v}=\curl\div(\dev\boldsymbol{\tau})\in\curl\div H^{-1}(\curl\div, \Omega ; \mathbb{T})$.

\smallskip
\step 2 \; $\dev\curl H(\operatorname{curl}, \Omega ; \mathbb{M}) = H^{-1}(\curl\div, \Omega ; \mathbb{T})\cap\ker(\curl\div)$.

For $\boldsymbol{\tau}\in H^{-1}(\curl\div, \Omega ; \mathbb{T})\cap\ker(\curl\div)$, by the de Rham complex, there exists a function $u\in L^2(\Omega)$ s.t. $\div\boldsymbol{\tau}=\grad u$. Then $\div(\boldsymbol{\tau}-u\boldsymbol{I})=0$, which means $\boldsymbol{\tau}=u \boldsymbol{I}+\curl\boldsymbol{\sigma}$ with $\boldsymbol{\sigma}\in H^1(\Omega ; \mathbb{M})$. By the traceless of $\boldsymbol{\tau}$, we get $\boldsymbol{\tau}=\dev\curl\boldsymbol{\sigma}\in \dev\curl H^1(\Omega ; \mathbb{M})\subseteq\dev\curl H(\operatorname{curl}, \Omega ; \mathbb{M})$.

\smallskip
\step 3 \; $ \grad H^{1}(\Omega ; \mathbb{R}^{3}) \oplus \textrm{span}\{\mskw\boldsymbol{x}\} = H(\operatorname{curl}, \Omega ; \mathbb{M})\cap\ker(\dev\curl)$.

Since $\curl(\mskw\boldsymbol{x})=2\boldsymbol{I}$, we have $\grad H^{1}(\Omega ; \mathbb{R}^{3}) \cap \textrm{span}\{\mskw\boldsymbol{x}\}=\{0\}$.
For $\boldsymbol{\tau}\in H(\operatorname{curl}, \Omega ; \mathbb{M})\cap\ker(\dev\curl)$, we have $\curl\boldsymbol{\tau}=\frac{1}{3}\tr(\curl\boldsymbol{\tau})\boldsymbol{I}$. Apply $\div$ on both sides to get $\grad(\tr(\curl\boldsymbol{\tau}))=0$. Then $\tr(\curl\boldsymbol{\tau})$ is constant, and
$\curl\boldsymbol{\tau}=2c\boldsymbol{I}$ with $c\in\mathbb R$. This implies $\curl(\boldsymbol{\tau} -c\mskw\boldsymbol{x})=0$. Therefore $\boldsymbol{\tau}\in \grad H^{1}(\Omega ; \mathbb{R}^{3}) \oplus \textrm{span}\{\mskw\boldsymbol{x}\}$.
\end{proof}

Next we use the framework developed in~\cite{ChenHuang2018} to present a Helmholtz decomposition of $H^{-1}(\curl\div,\Omega;\mathbb T)$. Denote by \begin{equation*}K^c = H_0(\curl,\Omega)\cap \ker(\div) = \{\bs v\in H_0(\curl,\Omega): \bs v\perp \grad H_0^1(\Omega)\}.\end{equation*}
Then $\curl\curl: K^c\to H^{-1}(\div,\Omega)\cap\ker(\div)$ is isomorphic. Indeed, by the Helmholtz decomposition $L^2(\Omega;\mathbb R^3)=\curl K^c+\nabla H^1(\Omega)$~\cite{ArnoldFalkWinther2006}, we have 
\begin{equation*}
\curl\curl K^c=\curl L^2(\Omega;\mathbb R^3)= H^{-1}(\div,\Omega)\cap\ker(\div).
\end{equation*}

\begin{lemma}
It holds the Helmholtz decomposition
\begin{equation*}
H^{-1}(\curl\div,\Omega;\mathbb T)=\dev\curl H(\curl,\Omega;\mathbb M)\oplus \mskw K^c.
\end{equation*}
\end{lemma}
\begin{proof}
With complex \eqref{distribcurldivcomplex} and identity \eqref{eq:divmskw}, we build up the commutative diagram
\begin{equation*}
\begin{array}{c}
\xymatrix@R-=1.6pc{
H(\curl,\Omega;\mathbb M) \ar[r]^-{\dev\curl} & H^{-1}(\curl\div,\Omega;\mathbb T)\ar[r]^-{\curl\div} & H^{-1}(\div,\Omega)\cap\ker(\div)\to0  \\
& &K^c \ar[u]_{\curl\curl}\ar[lu]^{-\mskw}.
} 
\end{array}
\end{equation*}
Apply the framework in~\cite{ChenHuang2018} to get the required Helmholtz decomposition.
\end{proof}

We introduce the tensor space with tangential-normal continuity
\begin{align*}
\Sigma^{\rm tn}:=\{\boldsymbol{\tau}  \in H^{1}(\mathcal T_h ; \mathbb{T}) :  \left.\llbracket\boldsymbol{n}\times\boldsymbol{\tau}\boldsymbol{n} \rrbracket\right|_{F}=0 \text { for each } F \in \mathring{\mathcal{F}}_{h}\},
\end{align*}
where $\llbracket\boldsymbol{n}\times\boldsymbol{\tau}\boldsymbol{n} \rrbracket$ is the jump of $\boldsymbol{n}\times\boldsymbol{\tau}\boldsymbol{n}$ across $F$.
Let space 
$
V_0^{\curl} := H_0(\curl,\Omega)\cap H^1(\curl,\mathcal T_h), 
$
where $H^1(\curl,\mathcal T_h):=\{\boldsymbol{v}\in H^1(\mathcal T_h;\mathbb R^3): \curl\boldsymbol{v}\in H^1(\mathcal T_h;\mathbb R^3)\}$.
We define a weak operator
$
(\curl\div)_w: \Sigma^{\rm tn} \to (V_0^{\curl})' \subset (C_0^{\infty}(\Omega; \mathbb R^3))'
$
by
\begin{equation}\label{eq:weakcurldiv}
\langle (\curl\div)_w \boldsymbol{\tau}, \boldsymbol{v}\rangle :=\sum_{T\in \mathcal{T}_{h}}(\operatorname{div} \boldsymbol{\tau}, \curl\boldsymbol{v})_{T}-\sum_{F\in \mathring{\mathcal{F}}_{h}}(\llbracket \boldsymbol{n}^{\intercal}\boldsymbol{\tau}\boldsymbol{n} \rrbracket, \boldsymbol{n}_F\cdot\curl\boldsymbol{v})_{F}.
\end{equation}
Notice that only interior faces are included in the second term as $\bs n_F\cdot \curl \bs v = \rot_F\bs v = 0$ for $\bs v\in H_0(\curl,\Omega)$. 

\begin{lemma}\label{lem:curldivw}
 For $\bs \tau \in \Sigma^{\rm tn}$, the following identity holds in the distribution sense
\begin{equation*}
(\curl\div)_w \boldsymbol{\tau}=\curl\div \boldsymbol{\tau}.
\end{equation*}
\end{lemma}
\begin{proof}
By the definition of the distributional derivative and employing the integration by parts element-wise we get for $\bs v\in C_0^{\infty}(\Omega;\mathbb R^3)$ that
\begin{align*}
\langle\curl\div \boldsymbol{\tau}, \boldsymbol{v}\rangle &:=-(\boldsymbol{\tau}, \grad\curl\boldsymbol{v})=\sum_{T\in \mathcal{T}_{h}}(\operatorname{div} \boldsymbol{\tau}, \curl\boldsymbol{v})_{T}-\sum_{T\in \mathcal{T}_{h}}(\boldsymbol{\tau}\boldsymbol{n}, \curl\boldsymbol{v})_{\partial T} \\
&\;=\sum_{T\in \mathcal{T}_{h}}(\operatorname{div} \boldsymbol{\tau}, \curl\boldsymbol{v})_{T}-\sum_{T\in \mathcal{T}_{h}}(\boldsymbol{n}^{\intercal}\boldsymbol{\tau}\boldsymbol{n}, \boldsymbol{n}\cdot\curl\boldsymbol{v})_{\partial T} \\
&\;\quad-\sum_{T\in \mathcal{T}_{h}}(\boldsymbol{n}\times\boldsymbol{\tau}\boldsymbol{n}, \boldsymbol{n}\times\curl\boldsymbol{v})_{\partial T}.
\end{align*}
As $\boldsymbol{n}\times\boldsymbol{\tau}\boldsymbol{n}$ is continuous and $\boldsymbol{v} \in C_0^{\infty}(\Omega;\mathbb R^3)$, the last term is canceled. Then rearrange the second term face-wisely to derive 
 \begin{equation*}
\langle\curl\div \boldsymbol{\tau}, \boldsymbol{v}\rangle = \langle (\curl\div)_w \boldsymbol{\tau}, \boldsymbol{v}\rangle, \quad \forall~\boldsymbol{v} \in C_0^{\infty}(\Omega;\mathbb R^3).
\end{equation*}
Thus, $(\curl\div)_w \boldsymbol{\tau}=\curl\div \boldsymbol{\tau}$ in the distribution sense.
\end{proof}

Similarly we can define the weak operator
$
(\grad \curl)_w: V_0^{\curl} \to (\Sigma^{\rm tn})'
$
as
\begin{equation*}
\langle (\grad \curl)_w \bs v, \bs \tau \rangle : = \sum_{T\in \mathcal{T}_{h}}(\boldsymbol{\tau}, \grad\curl\boldsymbol{v})_{T} - \sum_{F\in \mathcal{F}_{h}}(\boldsymbol{n}\times\boldsymbol{\tau}\boldsymbol{n}, \llbracket \boldsymbol{n}\times\curl\boldsymbol{v}\rrbracket )_{F}.
\end{equation*}
By definition, we have the duality
\begin{equation}\label{eq:duality}
\langle(\curl\div)_w \boldsymbol{\tau}, \boldsymbol{v}\rangle  = - \langle \bs \tau, (\grad \curl)_w \bs v \rangle, \quad \bs \tau \in \Sigma^{\rm tn}, \bs v\in V_0^{\curl}.
\end{equation}

When $\bs \tau \in H(\curl\div,\Omega; \mathbb T)\cap \Sigma^{\rm tn}$, $\langle(\curl\div)_w \boldsymbol{\tau}, \boldsymbol{v}\rangle = (\curl\div\bs \tau, \bs v)$ and when $\bs v\in H(\grad\curl,\Omega)$, $\langle \bs \tau, (\grad \curl)_w \bs v \rangle = (\bs \tau, \grad \curl \bs v)$. The duality \eqref{eq:duality} strikes a balance of the smoothness of $\bs \tau$ and $\bs v$ so that the second order differential operators can be defined for less smooth functions.

\section{Distributional finite element $\curl\div$ complex}\label{sec:distribfecurldivcomplex}
We shall construct a finite element counterpart of the distributional $\curl\div$ complex~\eqref{distribcurldivcomplex}. 

\subsection{Finite element spaces}
We first recall the tangential-normal continuous finite element for traceless tensors in \cite{GopalakrishnanLedererSchoeberl2020a}.
Take $\mathbb P_{k}(T;\mathbb T)$ as the space of shape functions with $k\geq 0$. The degrees of freedom (DoFs) are given by
\begin{subequations}\label{eq:curldivdof}
 \begin{align}\label{eq:curdivfemdof1}
\int_{F} \boldsymbol{t}_i^{\intercal}\boldsymbol{\tau}\boldsymbol{n}\, q \dd S, \quad & q \in \mathbb P_{k}(F), i = 1,2, F\in\mathcal F(T), \\
\label{eq:curdivfemdof2}
	\int_{T} \boldsymbol{\tau}: \boldsymbol{q} \dx, \quad & \boldsymbol{q} \in \mathbb{P}_{k-1}(T; \mathbb{T}).
\end{align}
\end{subequations}

In order to give a geometric decomposition of space $\mathbb P_{k}(T;\mathbb T)$, we present two intrinsic bases of $\mathbb T$ which are variants of a basis constructed in~\cite{HuLiang2021}. 

\begin{figure}[htbp]
\label{fig:Tbasis}
\subfigure[Basis $\{\dev(\nabla \lambda_{i}\otimes\boldsymbol{t}_{i \ell}), \dev(\nabla \lambda_{j}\otimes\boldsymbol{t}_{j \ell})\}_{\ell= 0}^3
$.]{
\begin{minipage}[t]{0.54\linewidth}
\centering
\includegraphics*[width=5.0cm]{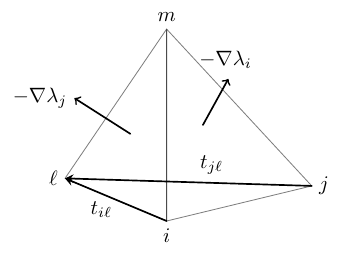}
\end{minipage}}
\subfigure[Basis $\{\boldsymbol{t}_{mi}\otimes\nabla\lambda_{\ell}, \boldsymbol{t}_{mj}\otimes\nabla\lambda_{\ell}\}_{\ell=0}^3$.]
{\begin{minipage}[t]{0.45\linewidth}
\centering
\includegraphics*[width=4.54cm]{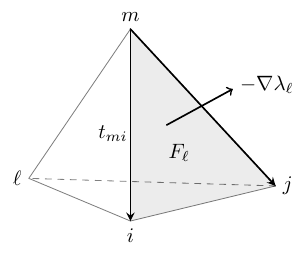}
\end{minipage}}
\caption{Two intrinsic bases of traceless matrix $\mathbb T$.}
\end{figure}

\begin{lemma}\label{lem:Tbasis}
Let $(ij\ell m)$ be a cyclic permutation of $(0123)$. Then the set 
\begin{equation}\label{eq:basis}
\{\dev(\nabla \lambda_{i}\otimes\boldsymbol{t}_{i \ell}), \dev(\nabla \lambda_{j}\otimes\boldsymbol{t}_{j \ell}), \ell= 0,\ldots, 3\}
\end{equation}
is dual to 
\begin{equation}\label{eq:dualbasis}
\{\boldsymbol{t}_{mi}\otimes\nabla\lambda_{\ell}, \boldsymbol{t}_{mj}\otimes\nabla\lambda_{\ell}\}_{\ell=0}^3.
\end{equation} 
Consequently both are bases of $\mathbb T$. 
\end{lemma}
\begin{proof}

The duality follows from the identity $\boldsymbol{t}_{ij}\cdot\nabla\lambda_{\ell}=\delta_{j\ell}-\delta_{i\ell}$, where $\delta_{i\ell}$ is the Kronecker delta function, and $\boldsymbol{t}\otimes\nabla\lambda_{\ell} : \bs I = \bs t\cdot \nabla \lambda_{\ell} = 0$ for vector $\boldsymbol{t}$ tangent to $F_{\ell}$. 
\end{proof}

As $\nabla \lambda_i \parallel \bs n_{F_i}$, the basis \eqref{eq:dualbasis} is face-wise and each face contributes two while the basis~\eqref{eq:basis} is vertex-wise. They are illustrated in Fig. \ref{fig:Tbasis}.


Define 
\begin{equation*}
\tr_F^{\rm tn} \bs \tau = \bs n_F\times \bs \tau \bs n_F,
\end{equation*}
and $\tr^{\rm tn}: C(T; \mathbb T)\to L^2(\partial T;\mathbb R^2)$ as $\tr^{\rm tn}|_F = \tr_F^{\rm tn}$. 
Let the bubble polynomial space of degree $k$
\begin{equation*}
\mathbb B_{k}^{\rm tn}(T;\mathbb T):=\{\boldsymbol{\tau}\in\mathbb P_{k}(T;\mathbb T): \tr^{\rm tn} \bs \tau =0\}.
\end{equation*}
Notice that for the identity matrix $\bs I$, $\tr^{\rm tn} \bs I = 0$ but $\bs I\not\in \mathbb B_{k}^{\rm tn}(T;\mathbb T)$ as ${\rm trace} (\bs I)\neq 0$.

\begin{lemma}\label{lem:Tbasis0}
For $\ell,i = 0,\ldots, 3, i\neq \ell$, we have
\begin{equation*}
\lambda_{\ell}\dev(\nabla \lambda_{i}\otimes\boldsymbol{t}_{i \ell})\in  \mathbb B_{k}^{\rm tn}(T;\mathbb T). 
\end{equation*}
\end{lemma}
\begin{proof}
Let $(ij\ell m)$ be a cyclic permutation of $(0123)$. The edge $\texttt{v}_{i}\texttt{v}_{\ell}$ is contained in faces $F_{m}$ and $F_j$ and thus $\boldsymbol{t}_{i \ell}\cdot \bs n_F = 0$ for $F=F_m, F_{j}$. The vector $\nabla \lambda_i \parallel \bs n_{F_i}$ and thus $\bs n_{F_i}\times \nabla \lambda_i = 0$ on $F_i$. On the face $F_{\ell}$, $\lambda_{\ell}|_{F_{\ell}} = 0$. Notice that the identity matrix $\bs I$ satisfies $\tr^{\rm tn}\bs I = 0$. 

So we have verified $\tr^{\rm tn}(\lambda_{\ell}\dev(\nabla \lambda_{i}\otimes\boldsymbol{t}_{i \ell})) = 0$. 
\end{proof}

By changing $\nabla \lambda_i$ to the parallel vector $\bs n_i$, we present the following geometric decomposition of $\mathbb P_{k}(T;\mathbb T)$.
\begin{lemma}
We have the geometric decomposition 
\begin{align*}
\mathbb P_{k}(T;\mathbb T)&=\Oplus_{\ell=0}^3\big(\mathbb P_{k}(F_{\ell})\otimes\mathrm{span}\{\dev(\bs n_i\otimes\boldsymbol{t}_{i \ell}), \dev(\bs n_j\otimes\boldsymbol{t}_{j \ell}) \}\big) \Oplus \mathbb B_{k}^{\rm tn}(T;\mathbb T),
\end{align*}
where we give a characterization of the bubble space
\begin{equation*}
\mathbb B_{k}^{\rm tn}(T;\mathbb T)= \mathbb P_{k-1}(T)\otimes\mathrm{span}\{\lambda_{\ell}\dev(\boldsymbol{n}_i\otimes\boldsymbol{t}_{i\ell}), \lambda_{\ell}\dev(\boldsymbol{n}_j\otimes\boldsymbol{t}_{j\ell}), \ell =0,\ldots, 3\}.
\end{equation*}
\end{lemma}


Define the global finite element space
\begin{align*}
{\Sigma}_{k,h}^{\rm tn}:=\{\boldsymbol{\tau}_h \in L^{2}(\Omega; \mathbb{T})&:\boldsymbol{\tau}_h|_T \in \mathbb{P}_{k}(T; \mathbb{T}) \textrm{ for each } T \in \mathcal{T}_{h}, \\
&\quad\text{and all the DoFs \eqref{eq:curdivfemdof1}-\eqref{eq:curdivfemdof2} are single-valued}\}.
\end{align*}

\begin{remark}\rm
The construction can be readily extended to arbitrary dimension $\mathbb R^d$ for $d\geq 2$. There are $d+1$ faces for a $d$-simplex. At each face, we have $d-1$ linearly independent traceless matrices $\{\bs t_{F,i}\otimes \bs n_{F}\}_{i=1}^{d-1}$ and in total $(d+1)(d-1)$ such matrices form a basis of $\mathbb T$. The basis of bubble functions is constructed vertices-wisely $\{\lambda_{\ell}\dev(\nabla \lambda_{i}\otimes\boldsymbol{t}_{i \ell})\}$ for $\ell = 0,\ldots, d$ and $d-1$ different $i$ for each $\ell$.  \qed
%
 \end{remark}

\begin{remark}\rm
Since the basis functions are expressed in terms of intrinsic geometric quantities—tangential and normal vectors along with barycentric coordinates—the defined finite element spaces are affine invariant, meaning the space remains unchanged under affine transformations. Thus, the traditional method of using a reference tetrahedron and transforming to the physical element can be applied, resulting in the same finite element space. In particular, curved elements can be handled using the Piola transform, as described in equations (3.76)-(3.77) of \cite{Monk2003}. Extending the two-dimensional work of \cite{ArnoldWalker2020} to a distributional mixed finite element method specifically tailored for the quad-curl problem on domains with curved boundaries is an interesting topic for future research.
\qed \end{remark}

As $H^{-1}(\operatorname{div}, \Omega) = (H_0(\curl,\Omega))'$, we can use $H(\curl)$-conforming finite elements, i.e. N\'ed\'elec elements~\cite{Nedelec1980,Nedelec1986} for the pair space. Take 
\begin{equation*}
\mathcal N_{k,\ell}^c(T):=\boldsymbol{x}\times\mathbb P_{k-1}(T;\mathbb R^3) + \grad\mathbb P_{\ell+1}(T), \quad \ell = k \text{ or } k -1
\end{equation*}
 as the space of shape functions. The element $\mathcal N_{k,k-1}^c$ is the first kind and $\mathcal N_{k, k}^c$ is the second kind N\'ed\'elec elements~\cite{Nedelec1980,Nedelec1986}. 
The degrees of freedom $\mathcal N_{k,\ell}^c(T)$ are given by
\begin{subequations}\label{eq:Nedelec}
\begin{align}
(\boldsymbol v\cdot\boldsymbol t, q)_e, & \quad q\in\mathbb P_{\ell}(e),  e\in\mathcal E(T),\label{Hcurlfem3ddof1}\\
(\boldsymbol n\times\boldsymbol v\times\boldsymbol{n}, \bs q)_F, & \quad \boldsymbol q\in\curl_F\mathbb P_{
k-1}(F)\oplus\bs x\mathbb P_{\ell-2}(F),  F\in\mathcal F(T),\label{Hcurlfem3ddof2}\\
(\boldsymbol v, \boldsymbol q)_T, & \quad\boldsymbol q\in\curl\mathbb P_{k-2}(T; \mathbb R^3)\oplus\boldsymbol x\mathbb P_{\ell-3}(T). \label{Hcurlfem3ddof3}
\end{align}
\end{subequations}

Define global finite element spaces
\begin{align*}
\mathbb V_{(k,\ell), h}^{\curl}:=\{\boldsymbol{v}_h\in H(\curl,\Omega): &\boldsymbol{v}_h|_T \in \mathcal N_{k,\ell}^c(T) \textrm{ for } T \in \mathcal{T}_{h} \\
&\text{ and all the DoFs \eqref{Hcurlfem3ddof1}-\eqref{Hcurlfem3ddof3} are single-valued}\}.
\end{align*}
Let $\mathring{\mathbb V}_{(k,\ell),h}^{\curl}:=\mathbb V_{(k,\ell),h}^{\curl}\cap H_0(\curl,\Omega)$ and $\mathbb V_{k, h}^{\curl}:=\mathbb V_{(k,k), h}^{\curl}$.

We use the standard Lagrange element for $H^1(\Omega)$
\begin{equation*}
\mathbb V_{\ell+1, h}^{\grad}:=\{\psi_h\in H^1(\Omega): \psi_h|_T \in\mathbb P_{\ell+1}(T) \textrm{ for } T \in \mathcal{T}_{h}\},
\end{equation*}
 and let $\mathring{\mathbb V}_{\ell+1,h}^{\grad}:=\mathbb V_{\ell+1,h}^{\grad}\cap H_0^1(\Omega)$. 
Let
\begin{equation*}
\mathbb V_{k+1,h}^{\grad}(\mathbb{R}^{3}):=\mathbb V_{k+1,h}^{\grad}\otimes\mathbb{R}^{3}, \quad \text{ and } \quad \mathbb V_{k,h}^{\curl}(\mathbb{M}):=\mathbb{R}^{3}\otimes\mathbb V_{k,h}^{\curl}.
\end{equation*}

The degree of polynomial may be skipped in the notation of finite element spaces when it is clear from the context. 

\subsection{Distributional finite element complex}

By treating the right hand side of \eqref{eq:weakcurldiv} as a bilinear form defined on $\bs \Sigma^{\rm tn}\times V_0^{\curl}$, the weak operators can be naturally extended to the discrete spaces by restricting the bilinear form to subspaces. Define
\begin{align*}
&(\curl\div)_h = J_h(\curl\div)_w: {\Sigma}_{h}^{\rm tn} \to (\mathring{\mathbb V}_{h}^{\curl})' \cong \mathring{\mathbb V}_{h}^{\curl}, \\
&(\grad\curl)_h = J_h(\grad\curl)_w:  \mathring{\mathbb V}_{h}^{\curl} \to ({\Sigma}_{h}^{\rm tn} )'\cong {\Sigma}_{h}^{\rm tn},
\end{align*}
where the isomorphism $J_h$ is the Reisz representation of the $L^2$-inner product and realized by the inverse of the mass matrix. More precisely, for $\bs \sigma_h\in \Sigma_h^{\rm tn}$, $(\curl\div)_h\bs{\sigma}_h\in \mathring{\mathbb V}_{h}^{\curl}$ such that
\begin{align}
\notag
((\curl\div)_h\bs{\sigma}_h, \bs v_h) &=\langle(\curl\div)_w\bs{\sigma}_h, \bs v_h\rangle\\
\label{eq:curldivhcurldivw}
&= \sum_{T\in \mathcal{T}_{h}}(\operatorname{div} \boldsymbol{\sigma}_h, \curl\boldsymbol{v}_h)_{T}-\sum_{F\in \mathring{\mathcal{F}}_{h}}(\llbracket \boldsymbol{n}^{\intercal}\boldsymbol{\sigma}_h\boldsymbol{n} \rrbracket, \boldsymbol{n}_F\cdot\curl\boldsymbol{v}_h)_{F}\\
\notag
&= - \sum_{T\in \mathcal{T}_{h}}(\boldsymbol{\sigma}_h, \grad\curl\boldsymbol{v}_h)_{T}+ \sum_{F\in\mathcal F_h}(\bs n\times\bs\sigma_h\bs n, \llbracket \bs n\times \curl\bs v_h\rrbracket)_F\\
\notag
&=-(\boldsymbol{\sigma}_h, (\grad\curl)_h\bs v_h).
\end{align}
For a fixed triangulation $\mathcal T_h$, $(\curl\div)_h$ is well defined which can be obtained by inverting the mass matrix of $\mathring{\mathbb V}_h^{\curl}$. However $\{(\curl\div)_h\}$ is not uniformly bounded when $h\to 0$ as \eqref{eq:curldivhcurldivw} is not a bounded linear functional of $L^2(\Omega)$. 



Similarly define discrete div operator $\div_h: H_0(\curl,\Omega)\to \mathring{\mathbb V}_h^{\grad}$ by 
\begin{equation*}
(\div_h\boldsymbol{v}, \psi_h) = -(\boldsymbol{v}, \grad \psi_h),\quad \psi_h\in \mathring{\mathbb V}_h^{\grad}.
\end{equation*}
Notice that $\div_h$ is the $L^2$-adjoint of $-\grad$ restricted to $\mathring{\mathbb V}_h^{\grad}$. Again for a fixed triangulation $\mathcal T_h$, $\div_h$ is well defined which can be obtained by inverting the mass matrix of $\mathring{\mathbb V}_h^{\grad}$. However $\{\div_h\}$ is not uniformly bounded as $h\to 0$ as $(\bs v,\grad \psi)$ is in general not a bounded linear functional of $L^2$ unless $\bs v\in H(\div,\Omega)$.

\begin{theorem}
The distributional finite element $\curl\div$ complex is
\begin{equation}\label{femcurldivcomplex}
\begin{aligned}
\mathbb{R}^{3} \times\{0\} \rightarrow  \mathbb V_{k+1,h}^{\grad}(\mathbb{R}^{3}) \times \mathbb{R} &\xrightarrow{(\grad\!, \,\mskw\boldsymbol{x})} \mathbb V_{k,h}^{\curl}(\mathbb{M})\xrightarrow{\dev\curl}\\ 
\Sigma_{k-1,h}^{\rm tn}
& \xrightarrow{(\curl\div)_h} \mathring{\mathbb V}_{(k,\ell), h}^{\curl} \xrightarrow{\div_h} \mathring{\mathbb V}_{\ell+1, h}^{\grad} \rightarrow 0.
\end{aligned}
\end{equation}
When $\Omega\subset\mathbb R^3$ is a bounded and topologically trivial Lipschitz domain, \eqref{femcurldivcomplex} is exact.
\end{theorem}
\begin{proof}
 As $\curl\grad = 0$, it is straightforward to verify $\div_h(\curl\div)_h = 0$. Take $\boldsymbol{\tau}_h=\dev\curl\boldsymbol{\sigma}_h$ with $\boldsymbol{\sigma}_h\in\mathbb V_h^{\curl}(\mathbb{M})$. Since $\boldsymbol{n}\times\boldsymbol{\tau}_h \boldsymbol{n}=\boldsymbol{n}\times(\curl\boldsymbol{\sigma}_h)\boldsymbol{n}$ is single-value across each face $F\in\mathring{\mathcal F}_h$, we get $\boldsymbol{\tau}_h\in{\Sigma}^{\rm tn}_h$. In the distribution sense, $\curl\div \dev \curl = 0$ and so is $(\curl\div)_h \dev \curl $. In summary, we have verified \eqref{femcurldivcomplex} is a complex.

Verification of $\ker (\grad, \mskw \bs x) = \mathbb{R}^{3} \times\{0\}$ is trivial. For $\bs \tau_h \in \mathbb V_{k,h}^{\curl}(\mathbb{M})$ and $\dev\curl (\bs \tau_h) = 0$, by the exactness of \eqref{distribcurldivcomplex}, we can find $(\bs v,c)\in H^1(\Omega;\mathbb R^3)\times \mathbb R$ s.t. $\grad \bs v + c\mskw \bs x = \bs \tau_h$. As $\bs \tau_h$ is polynomial of degree $k$, we conclude $\bs v\in \mathbb V_{k+1,h}^{\grad}(\mathbb{R}^{3})$.   

From the finite element de Rham complex $\grad\mathring{\mathbb V}_{\ell+1, h}^{\grad}=\mathring{\mathbb V}_{(k,\ell), h}^{\curl}\cap\ker(\curl)$, we have
\begin{equation}\label{eq:divhonto}
\div_h\mathring{\mathbb V}_{(k,\ell),h}^{\curl}=\mathring{\mathbb V}_{\ell+1, h}^{\grad}.
\end{equation}

It remains to prove
\begin{align}
\label{eq:kercurldiv}\dev\curl\mathbb V_{k,h}^{\curl}(\mathbb{M})&=\Sigma_{k-1,h}^{\rm tn}\cap\ker((\curl\div)_h),\\
\label{eq:kerdivh}(\curl\div)_h \Sigma_{k-1,h}^{\rm tn} &= \mathring{\mathbb V}_{(k,\ell),h}^{\curl} \cap \ker (\div_h).
\end{align}
We will prove \eqref{eq:kerdivh} in Corollary \ref{cor:kerdivh}  and \eqref{eq:kercurldiv} in Corollary \ref{cor:kercurldivh}.
\end{proof}


\subsection{Characterization of null spaces}
Define $K_{h}^c=\mathring{\mathbb V}_{(k,\ell),h}^{\curl} \cap \ker (\div_h)$ and $(\curl\curl)_h: K_{h}^c \to K_{h}^c$ so that 
\begin{equation}\label{eq:curcurlh}	
((\curl\curl)_h\boldsymbol{u}_h, \boldsymbol{v}_h)=(\curl\boldsymbol{u}_h, \curl\boldsymbol{v}_h),\quad \boldsymbol{v}_h\in K_h^c.
\end{equation}

\begin{lemma}
We have the discrete Poincar\'e inequalities
\begin{align}
\label{eq:curlhpoincare}	
\|\bs v_h\| &\lesssim  \|\curl\boldsymbol{v}_h\|, \quad\quad\quad\;\;\;\, \boldsymbol{v}_h\in K_h^c,\\
\label{eq:curlcurlhpoincare}	
\|\boldsymbol{v}_h\|_{H(\curl)}&\lesssim \|(\curl\curl)_h\boldsymbol{v}_h\|,\quad \boldsymbol{v}_h\in K_h^c.
\end{align}
\end{lemma}
\begin{proof}
The first Poincar\'e inequality \eqref{eq:curlhpoincare} can be found in~\cite[Lemma 3.4 and Theorem 3.6]{GiraultRaviart1986} and~\cite[Lemma 7.20]{Monk2003}.
Consequently $(\curl\cdot,\curl\cdot)$ is an inner product on $K_h^c$ and the operator $(\curl\curl)_h$ is isomorphic.

Taking $\boldsymbol{u}_h=\boldsymbol{v}_h$ in \eqref{eq:curcurlh} and applying \eqref{eq:curlhpoincare}, we get
\begin{equation*}
\|\curl\boldsymbol{v}_h\|^2\leq \|(\curl\curl)_h\boldsymbol{v}_h\|\|\boldsymbol{v}_h\|\lesssim \|(\curl\curl)_h\boldsymbol{v}_h\|\|\curl\boldsymbol{v}_h\|. 
\end{equation*}
Hence $\|\curl\boldsymbol{v}_h\|\lesssim \|(\curl\curl)_h\boldsymbol{v}_h\|.$
The proof is finished by applying \eqref{eq:curlhpoincare} again to bound $\|\bs v_h\|$.
\end{proof}


We will introduce interpolation operators satisfying the following commutative diagrams: 
\begin{figure}[htbp]
\label{fig:interpolation}
\centering
\includegraphics*[width=4cm]{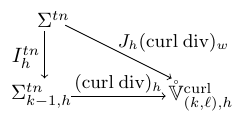}
\qquad
\includegraphics*[width=4cm]{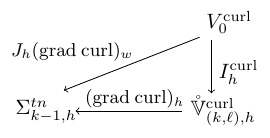}
\caption{Identities connecting the weak differential operators and interpolation operators.}
\end{figure}

Let $I_h^{\rm tn}:\Sigma^{\rm tn}\to {\Sigma}^{\rm tn}_{k-1,h}$ be the interpolation operator using DoF \eqref{eq:curldivdof} and denote by $\bs \sigma_I = I_h^{\rm tn} \bs \sigma$. 

\begin{lemma}
For $\boldsymbol{\tau}\in \Sigma^{\rm tn}$ and $\boldsymbol{v}_h\in \mathring{\mathbb V}_{(k,\ell),h}^{\curl}$, it holds
\begin{equation}\label{eq:curdivhcdprop}
((\curl\div)_h(I_h^{\rm tn}\boldsymbol{\tau}), \bs v_h)=\langle (\curl\div)_w\boldsymbol{\tau}, \bs v_h \rangle.
\end{equation}	
\end{lemma}
\begin{proof}
By definition,
\begin{equation*}
\sum_{T\in \mathcal{T}_{h}}-(\boldsymbol{\tau}-I_h^{\rm tn}\boldsymbol{\tau}, \grad\curl\boldsymbol{v}_h)_T + (\boldsymbol{n}\times(\boldsymbol{\tau}-I_h^{\rm tn}\boldsymbol{\tau})\boldsymbol{n}, \boldsymbol{n}\times\curl\boldsymbol{v}_h)_{\partial T}=0,
\end{equation*}
which is true due to DoFs \eqref{eq:curdivfemdof1}-\eqref{eq:curdivfemdof2} and the fact $\grad\curl\boldsymbol{v}_h|_T\in \mathbb P_{k-2}(T;\mathbb R^3)$ and $\boldsymbol{n}\times\curl\boldsymbol{v}_h|_F\in \mathbb P_{k-1}(F;\mathbb R^2)$.
\end{proof}

We introduce the interpolation operator $I_h^{\curl}: V_0^{\curl}\to \mathring{\mathbb V}_{(k,\ell),h}^{\curl}$ defined by DoF~\eqref{eq:Nedelec} and denote by $\bs v_I = I_h^{\curl}\bs v$.

\begin{lemma}\label{lem:interpolation}
For $\bs v\in V_0^{\curl}$ and $\boldsymbol{\tau}_h\in\Sigma^{\rm tn}_{k-1,h}$, it holds
\begin{equation}\label{eq:curdivhIhcurl}
\langle (\curl\div)_w\boldsymbol{\tau}_h, \bs v\rangle = ((\curl\div)_h \boldsymbol{\tau}_h, I_h^{\curl} \bs v).
\end{equation}
\end{lemma}
\begin{proof}
By integration by parts and $\bs \tau_h\mid_T \in \mathbb P_{k-1}(T; \mathbb{T})$, we have
\begin{align*}
&\quad\langle (\curl\div)_w\boldsymbol{\tau}_h, \bs v \rangle \\
= &\sum_{T\in \mathcal T_h} (\curl \div \bs \tau_h, \bs v)_T - \sum_{F\in \mathring{\mathcal{F}}_{h}}([\bs n\times \div \bs \tau_h + \curl_F(\bs n^{\intercal}\bs \tau_h \bs n)], \bs v)_{F}\\
&- \sum_{e\in\mathring{\mathcal{E}}_{h}}([\boldsymbol{n}^{\intercal}\boldsymbol{\tau}_h\boldsymbol{n}]_e, \boldsymbol{v}\cdot\boldsymbol{t}_{e})_{e}\\
= &\sum_{T\in \mathcal T_h} (\curl \div \bs \tau_h, \bs v_I)_T  - \sum_{F\in \mathring{\mathcal{F}}_{h}}([\bs n\times \div \bs \tau_h + \curl_F(\bs n^{\intercal}\bs \tau_h \bs n)], \bs v_I)_{F}\\
&-\sum_{e\in\mathring{\mathcal{E}}_{h}}([\boldsymbol{n}^{\intercal}\boldsymbol{\tau}_h\boldsymbol{n}]_e, \boldsymbol{v}_I\cdot\boldsymbol{t}_{e})_{e} = ((\curl\div)_h \boldsymbol{\tau}_h, \bs v_I),
\end{align*} 
where
$\displaystyle
[\boldsymbol{n}^{\intercal}\boldsymbol{\tau}_h\boldsymbol{n}]_e=\sum_{T\in\mathcal T_h}\sum_{F\in\mathcal F(T), e\subset\partial F}(\boldsymbol{n}^{\intercal}\boldsymbol{\tau}_h|_T\boldsymbol{n})|_e(\boldsymbol{t}_e\cdot\boldsymbol{t}_{F,e}).
$
\end{proof}

\begin{corollary}\label{lem:curldivhkernel}
Let $\boldsymbol{\tau}_h\in{\Sigma}^{\rm tn}_{k-1,h}$ satisfy $(\curl\div)_h\boldsymbol{\tau}_h=0$. Then we have $\curl\div\boldsymbol{\tau}_h=0$ and $\boldsymbol{\tau}_h\in H(\operatorname{curl} \div, \Omega ; \mathbb{T})$.
\end{corollary}
\begin{proof}
Let $\boldsymbol{\tau}_h\in{\Sigma}^{\rm tn}_h$ satisfy $(\curl\div)_h \boldsymbol{\tau}_h=0$. By Lemma \ref{lem:interpolation} and Lemma~\ref{lem:curldivw}, $\curl\div\boldsymbol{\tau}_h=0$ in the distribution sense. 
\end{proof}

At first glance, the null space $\Sigma_{k-1,h}^{\rm tn}\cap\ker((\curl\div)_h)$ is larger than the null space $\Sigma_{k-1,h}^{\rm tn}\cap\ker(\curl\div)$ as the test function is $ \mathring{\mathbb V}_{(k,\ell),h}^{\curl}$ not $C_0^{\infty}$. However Lemma~\ref{lem:interpolation} implies that they are the same due to the design of finite element spaces and weak differential operators. This is in the same spirit of Hellan-Herrmann-Johnson (HHJ) element~\cite{Hellan1967,Herrmann1967,Johnson1973,ChenHuHuang2018a} in 2D.

\begin{lemma}
It holds
\begin{equation}\label{eq:curdivhskwcurlcurlh}
(\curl\div)_h(I_h^{\rm tn}(\mskw\boldsymbol{u}_h))=-(\curl\curl)_h\boldsymbol{u}_h,\quad \boldsymbol{u}_h\in K_h^c.
\end{equation}
\end{lemma}
\begin{proof}
Since $(\mskw\boldsymbol{u}_h)\boldsymbol{n}=\boldsymbol{u}_h\times\boldsymbol{n}$, it follows that $\mskw\boldsymbol{u}_h\in\Sigma^{\rm tn}$. By the fact $\div\mskw\boldsymbol{u}_h=-\curl\boldsymbol{u}_h$,
we have
\begin{equation*}
\langle (\curl\div)_w(\mskw\boldsymbol{u}_h),\boldsymbol{v_h}\rangle =(\div\mskw\boldsymbol{u}_h,\curl\boldsymbol{v_h})=-(\curl\boldsymbol{u}_h,\curl\boldsymbol{v_h}).
\end{equation*}
Then the result holds from \eqref{eq:curdivhcdprop}.
\end{proof}

With complex \eqref{femcurldivcomplex} and identity \eqref{eq:curdivhskwcurlcurlh}, we have the commutative diagram
\begin{equation}\label{eq:triangulardiagram}
\begin{array}{c}
\xymatrix@R-=1.6pc@C+0.5pc{
\mathbb V_{k,h}^{\curl}(\mathbb{M}) \ar[r]^-{\dev\curl} & \Sigma^{\rm tn}_{k-1,h} \ar[r]^-{(\curl\div)_h} & K_{h}^c\ar[r]^-{} &0  \\
& &K_{h}^c\ar[u]_{(\curl\curl)_h}\ar[lu]^{-I_h^{\rm tn}\mskw}.
} 
\end{array}
\end{equation}

\begin{corollary}\label{cor:kerdivh}
It holds
\begin{equation}\label{eq:curldivhonto}
(\curl\div)_h\Sigma_{k-1,h}^{\rm tn}=K_{h}^c=\mathring{\mathbb V}_{(k,\ell),h}^{\curl} \cap \ker (\div_h).
\end{equation}
\end{corollary}
\begin{proof}
It is straightforward to verify that
$(\curl\div)_h{\Sigma}^{\rm tn}_h\subseteq K_{h}^c$.
On the other side, take $\boldsymbol{w}_h\in K_{h}^c$.
Let $\boldsymbol{u}_h=(\curl\curl)_h^{-1}\boldsymbol{w}_h\in K_{h}^c$ and set $\boldsymbol{\tau}_h=-I_h^{\rm tn}(\mskw\boldsymbol{u}_h)\in{\Sigma}^{\rm tn}_h$.
By~\eqref{eq:curdivhskwcurlcurlh}, 
\begin{equation*}
(\curl\div)_h\boldsymbol{\tau}_h=(\curl\curl)_h\boldsymbol{u}_h=\boldsymbol{w}_h,
\end{equation*}
which ends the proof.
\end{proof}

\begin{corollary}\label{cor:kercurldivh}
We have
\begin{equation}\label{eq:lmkercurldiv}
\dev\curl\mathbb V_{k,h}^{\curl}(\mathbb{M})=\Sigma_{k-1,h}^{\rm tn}\cap\ker((\curl\div)_h).
\end{equation}
\end{corollary}
\begin{proof}
By complex \eqref{femcurldivcomplex}, $\dev\curl\mathbb V_{k,h}^{\curl}(\mathbb{M})\subseteq\Sigma_{k-1,h}^{\rm tn}\cap\ker((\curl\div)_h)$. Then
we prove~\eqref{eq:lmkercurldiv} by dimension count. By~\eqref{eq:divhonto} and \eqref{eq:curldivhonto}, 
\begin{align*}
\dim\Sigma_{k-1,h}^{\rm tn}\cap\ker((\curl\div)_h)&=\dim\Sigma_{k-1,h}^{\rm tn}-\dim\mathring{\mathbb V}_{(k,k-1),h}^{\curl}+\dim\mathring{\mathbb V}_{k, h}^{\grad} \\
&=|\mathring{\mathcal V}_h|-|\mathring{\mathcal E}_h| + 2{k+1\choose2}|{\mathcal F}_h|-{k+1\choose2}|\mathring{\mathcal F}_h|+|\mathring{\mathcal F}_h| \\
&\quad+ |\mathcal T_h|\bigg(8{k+1\choose3}-3{k\choose3}+{k-1\choose3}\bigg).
\end{align*}
Similarly,
\begin{align*}
\dim\dev\curl\mathbb V_{k,h}^{\curl}(\mathbb{M})&=\dim\mathbb V_{k,h}^{\curl}(\mathbb{M})-\dim\mathbb V_{k+1,h}^{\grad}(\mathbb{R}^{3})+2 \\
&=-3|{\mathcal V}_h|+3|{\mathcal E}_h|+3{k+1\choose2}|{\mathcal F}_h|-3|{\mathcal F}_h|\\
&\quad+|{\mathcal T}_h|\bigg(9{k\choose3}-3{k-1\choose3}\bigg) +2.
\end{align*}
Combine the last two identities to get
\begin{align*}
&\quad \dim\Sigma_{k-1,h}^{\rm tn}\cap\ker((\curl\div)_h)-\dim\dev\curl\mathbb V_{k,h}^{\curl}(\mathbb{M}) \\
&=|\mathring{\mathcal V}_h|+3|{\mathcal V}_h|-|\mathring{\mathcal E}_h|-3|{\mathcal E}_h|+|\mathring{\mathcal F}_h|+3|{\mathcal F}_h|-4|{\mathcal T}_h|-2 \\
&\quad +{k+1\choose2}(4|{\mathcal T}_h|-|{\mathcal F}_h|-|\mathring{\mathcal F}_h|).
\end{align*}
Finally, we conclude the result from $4|{\mathcal T}_h|=|{\mathcal F}_h|+|\mathring{\mathcal F}_h|$, and the Euler's formulas $|{\mathcal V}_h|-|{\mathcal E}_h|+|{\mathcal F}_h|-|{\mathcal T}_h|=1$ and $|\mathring{\mathcal V}_h|-|\mathring{\mathcal E}_h|+|\mathring{\mathcal F}_h|-|{\mathcal T}_h|=-1$.
\end{proof}

\begin{remark}\rm 
Here is another proof of \eqref{eq:lmkercurldiv}. Take $\boldsymbol{\tau}_{h}\in\Sigma^{\rm tn}_{k-1,h}\cap\ker(\!(\curl\div)_h)$.  By Corollary~\ref{lem:curldivhkernel}, we have 
  $\curl\div\boldsymbol{\tau}_h=0$. Therefore $\boldsymbol{\tau}_h=\dev\curl\boldsymbol{\sigma}$ with $\boldsymbol{\sigma}\in H^1(\Omega ; \mathbb{M})$. Let $\boldsymbol{\sigma}_I\in \mathbb V_h^{\curl}(\mathbb{M})$ be the interpolation of $\boldsymbol{\sigma}$ based on DoFs \eqref{Hcurlfem3ddof1}-\eqref{Hcurlfem3ddof3}. The edge moment is not well defined for $H^1$ function but may be fixed by the fact $\dev\curl\boldsymbol{\sigma} = \bs \tau_h$ has extra smoothness. It is easy to verify that all the DoFs \eqref{eq:curdivfemdof1}-\eqref{eq:curdivfemdof2} of $\boldsymbol{\tau}_h-\dev\curl\boldsymbol{\sigma}_I$ vanish. Therefore $\boldsymbol{\tau}_h=\dev\curl\boldsymbol{\sigma}_I \in \dev\curl\mathbb V_h^{\curl}(\mathbb{M})$.   
\end{remark}


\subsection{Helmholtz decompositions}
The right half of the distributional finite element $\curl\div$ complex \eqref{femcurldivcomplex} is listed below
\begin{equation*}
\mathbb V_{k,h}^{\curl}(\mathbb{M})\xrightarrow{\dev\curl} \Sigma_{k-1,h}^{\rm tn}
 \xrightarrow{(\curl\div)_h} \mathring{\mathbb V}_{(k,\ell), h}^{\curl} \xrightarrow{\div_h} \mathring{\mathbb V}_{\ell+1, h}^{\grad} \rightarrow 0, \quad \ell = k \text{ or } k-1.
\end{equation*}
By taking the dual, we have the short exact sequence  
\begin{equation*}
0\leftarrow (\grad\curl)_h K_{h}^c
 \xleftarrow{(\grad\curl)_h} \mathring{\mathbb V}_{(k,\ell), h}^{\curl} = K_h^c \oplus \grad (\mathring{\mathbb V}_{\ell+1, h}^{\grad})  \xleftarrow{\grad} \mathring{\mathbb V}_{\ell+1, h}^{\grad} \leftarrow 0.
\end{equation*}


By the framework in~\cite{ChenHuang2018}, we get the following Helmholtz decompositions from the last two complexes, commutative diagram \eqref{eq:triangulardiagram} and \eqref{eq:curdivhskwcurlcurlh}.
\begin{corollary}
We have the discrete Helmholtz decompositions
\begin{align*}
\Sigma_{k-1,h}^{\rm tn} &=\dev\curl \mathbb V_{k,h}^{\curl}(\mathbb{M})\oplus^{L^2} (\grad\curl)_h K_{h}^c,\\
\Sigma_{k-1,h}^{\rm tn} &=\dev\curl \mathbb V_{k,h}^{\curl}(\mathbb{M})\oplus I_h^{\rm tn}(\mskw K_{h}^c).
\end{align*}
\end{corollary}

\begin{corollary}
We have the $L^2$-orthogonal Helmholtz decomposition of space $\mathring{\mathbb V}_{(k,\ell),h}^{\curl}$
\begin{align}\label{eq:helmholtz1}
\mathring{\mathbb V}_{(k,\ell),h}^{\curl} &= K_{h}^c  \oplus^{L^2}\grad \mathring{\mathbb V}_{\ell+1, h}^{\grad} = (\curl\curl)_hK_{h}^c \oplus^{L^2} \grad \mathring{\mathbb V}_{\ell+1, h}^{\grad}   \\
&= (\curl\div)_h\Sigma_{k-1,h}^{\rm tn} \oplus^{L^2}  \grad \mathring{\mathbb V}_{\ell+1, h}^{\grad} . \notag
\end{align}
\end{corollary}

\begin{lemma}
We have the discrete Poincar\'e	inequality
\begin{equation}\label{eq:gradcurlhpoincare}
\|\boldsymbol{v}_h\|_{H(\curl)}\lesssim \|(\grad\curl)_h\boldsymbol{v}_h\|,\quad \boldsymbol{v}_h\in K_h^c.
\end{equation}
\end{lemma}
\begin{proof}
Set $\bs \tau_h=I_h^{\rm tn}(\mskw\bs v_h)\in \Sigma_{k-1,h}^{\rm tn}$.
Then $(\curl\div)_h\bs \tau_h=-(\curl\curl)_h\bs v_h$ follows from \eqref{eq:curdivhskwcurlcurlh}. By the scaling argument, the inverse inequality and the Poincar\'e inequality
\eqref{eq:curlhpoincare},
\begin{equation}\label{eq:tauhnorm}
\|\bs \tau_h \| \lesssim\|\bs v_h\|\lesssim \|\curl\bs v_h\|.
\end{equation}

It follows that
\begin{equation*}
\|\curl\bs v_h\|^2=((\curl\curl)_h\bs v_h, \bs v_h)=-((\curl\div)_h\bs{\tau}_h, \bs v_h)=(\bs{\tau}_h, (\grad\curl)_h\bs v_h).
\end{equation*}
Applying Cauchy-Schwarz inequality and \eqref{eq:tauhnorm}, we obtain
\begin{equation*}
\|\curl\bs v_h\|\lesssim \|(\grad\curl)_h\bs v_h\|,
\end{equation*}
which implies \eqref{eq:gradcurlhpoincare}.
\end{proof}

\section{Mixed finite element method of the quad-curl problem}\label{sec:mfem}

Let $\Omega\subset\mathbb{R}^{3}$ be a bounded polygonal domain. Consider the fourth order problem
\begin{equation}\label{bigradcurlproblem}
	\begin{aligned}
		&\left\{\begin{aligned}
			-\curl\Delta\curl \boldsymbol{u}&=\boldsymbol{f} & & \text { in } \Omega, \\
			\operatorname{div} \boldsymbol{u} &=0 & & \text { in } \Omega, \\
			\boldsymbol{u} \times \boldsymbol{n}=\operatorname{curl} \boldsymbol{u}\times \boldsymbol{n} &=0 & & \text { on } \partial \Omega,
		\end{aligned}\right.
	\end{aligned}
\end{equation}
where $\boldsymbol{f} \in H^{-1}(\div, \Omega)\cap \ker(\div)$ is known. Such problem arises from multiphysics simulation such as modeling a
magnetized plasma in magnetohydrodynamics~\cite{kingsep1990reviews,Chacon;Simakov;Zocco:2007Steady-state}.

\subsection{Distributional mixed formulation}
Introducing $\boldsymbol{\sigma}:= \grad\curl\boldsymbol{u}$, we have $\tr\boldsymbol{\sigma} = \tr\grad\curl\boldsymbol{u} = \div\curl \bs u = 0$. Then rewrite problem \eqref{bigradcurlproblem} as the second-order system
\begin{equation}\label{quadcurl2ndsystem}
	\begin{aligned}
		\left\{\begin{aligned}
			\boldsymbol{\sigma} - \grad\curl \boldsymbol{u}&=0 \quad\quad \text { in } \Omega ,\\
			\curl{\div}	\boldsymbol{\sigma}&= -\boldsymbol{f}\quad \text { in } \Omega ,\\
			\div\boldsymbol{u} &= 0 \qquad \text { in } \Omega,\\
			\boldsymbol{u} \times \boldsymbol{n}=\operatorname{curl} \boldsymbol{u} &=0 \quad\;\;\;  \text { on } \partial \Omega.
		\end{aligned}\right.
	\end{aligned}
\end{equation}
A mixed formulation of the system \eqref{quadcurl2ndsystem} is to find $\boldsymbol{\sigma}\in H^{-1}(\operatorname{curl} \div, \Omega ; \mathbb{T})$, $\boldsymbol{u}\in H_{0}(\curl, \Omega)$ and $\phi\in H_0^{1}(\Omega)$ such that
\begin{subequations}
\begin{align}
	\label{quadcurlmixed1}	(\boldsymbol{\sigma},\boldsymbol{\tau})+b(\boldsymbol{\tau},\psi;\boldsymbol{u})&=0,  & & \boldsymbol{\tau} \in H^{-1}(\operatorname{curl} \div, \Omega ; \mathbb{T}), \psi\in H_0^{1}(\Omega), \\
	\label{quadcurlmixed2}	b(\boldsymbol{\sigma},\phi;\boldsymbol{v}) &=-\langle \boldsymbol{f}, \boldsymbol{v}\rangle, & &\boldsymbol{v} \in H_{0}(\curl,\Omega),
\end{align}
\end{subequations}
where the bilinear form $b(\cdot,\cdot; \cdot): (H^{-1}(\operatorname{curl} \div, \Omega ; \mathbb{T}) \times H_0^1(\Omega)) \times H_{0}(\curl, \Omega)$
\begin{equation*}
b(\boldsymbol{\tau},\psi;\boldsymbol{v}):=\langle\operatorname{curl} {\div}\boldsymbol{\tau}, \boldsymbol{v}\rangle+(\grad\psi, \boldsymbol{v}).
\end{equation*}
The term $(\grad\psi, \boldsymbol{u})$ is introduced to impose the divergence free condition $\div \bs u = 0$. 


\begin{lemma}
For $\boldsymbol{v} \in H_{0}(\curl,\Omega)$, it holds
\begin{equation}\label{eq:curldivinfsup}	
\|\boldsymbol{v}\|_{H(\curl)}\lesssim\sup_{\boldsymbol{\tau}\in H^{-1}(\operatorname{curl} \div, \Omega ; \mathbb{T}), \psi\in H_0^1(\Omega)}\frac{b(\boldsymbol{\tau},\psi;\boldsymbol{v})}{\|\boldsymbol{\tau}\|_{H^{-1}(\operatorname{curl} \div)}+|\psi|_1}.
\end{equation}
\end{lemma}
Proof of this lemma is similar to, indeed  simpler than, that of the discrete inf-sup condition cf. Lemma \ref{lm:discreteinfsup1}, we thus skip the details here.

\begin{lemma}\label{lm:nullspacecoercivity}
For $\boldsymbol{\tau} \in H^{-1}(\operatorname{curl} \div, \Omega ; \mathbb{T})$ and $\psi\in H_0^{1}(\Omega)$ satisfying 
\begin{equation}\label{eq:kernelBeqn}
b(\boldsymbol{\tau},\psi;\boldsymbol{v})=0,\quad \forall \ \boldsymbol{v} \in H_{0}(\curl,\Omega),
\end{equation}
it holds
\begin{equation}\label{eq:curldivcoercivity}
\|\boldsymbol{\tau}\|_{H^{-1}(\operatorname{curl} \div)}^2+|\psi|_1^2 = \|\boldsymbol{\tau}\|^2.
\end{equation}
\end{lemma}
\begin{proof}
By taking $\boldsymbol{v} = \grad \psi$ in \eqref{eq:kernelBeqn}, we get $\psi=0$. Then \eqref{eq:kernelBeqn} becomes
\begin{equation*}
\langle\operatorname{curl} {\div}\boldsymbol{\tau}, \boldsymbol{v}\rangle=0,\quad \forall \ \boldsymbol{v} \in H_{0}(\curl,\Omega).
\end{equation*}
Hence $\operatorname{curl} {\div}\boldsymbol{\tau}=0$, and \eqref{eq:curldivcoercivity} follows.
\end{proof}

Combining \eqref{eq:curldivinfsup}, \eqref{eq:curldivcoercivity} and the Babu{\v{s}}ka-Brezzi theory~\cite{BoffiBrezziFortin2013} yileds the wellposedness of the mixed formulation \eqref{quadcurlmixed1}-\eqref{quadcurlmixed2}.

%
\begin{theorem}
The mixed formulation \eqref{quadcurlmixed1}-\eqref{quadcurlmixed2}	is well-posed. Namely for any $\boldsymbol{f} \in H^{-1}(\div, \Omega)\cap \ker(\div)$, there exists a unique solution $(\bs \sigma, \bs u, \phi)$ to \eqref{quadcurlmixed1}-\eqref{quadcurlmixed2}. Furthermore we have $\phi=0$, and the stability
\begin{equation}\label{eq:sigmauboundedness}	
\|\boldsymbol{\sigma}\|_{H^{-1}(\operatorname{curl} \div)} + \|\boldsymbol{u}\|_{H(\curl)}\lesssim \|\boldsymbol{f}\|_{H^{-1}(\div)}.
\end{equation}
\end{theorem}
\begin{proof}
Combine the inf-sup condition \eqref{eq:curldivinfsup} and the coercivity \eqref{eq:curldivcoercivity} to get \eqref{eq:sigmauboundedness} and the well-posedness of the mixed formulation \eqref{quadcurlmixed1}-\eqref{quadcurlmixed2}. By choosing $\bs v=\grad\phi$ in \eqref{quadcurlmixed2}, we get $\phi= 0$ from $\div\bs f=0$.
\end{proof}

\subsection{Distributional mixed finite element method}
For $(\boldsymbol{\tau}, \psi)\in\Sigma^{\rm tn}\times H_0^1(\Omega)$ and $\boldsymbol{v}\in V_0^{\curl}$, introduce the bi-linear form
\begin{equation*}
b_h(\boldsymbol{\tau},\psi;\boldsymbol{v}):=\langle(\curl\div)_w\boldsymbol{\tau}, \boldsymbol{v}\rangle +(\grad\psi, \boldsymbol{v}).
\end{equation*}
By \eqref{eq:curldivhcurldivw},
we have for $(\boldsymbol{\tau}, \psi)\in\Sigma_{k-1,h}^{\rm tn}\times \mathring{\mathbb V}_{\ell+1, h}^{\grad}$ and $\boldsymbol{v}\in \mathring{\mathbb V}_{(k,\ell), h}^{\curl}$ that
\begin{equation*}
b_h(\boldsymbol{\tau},\psi;\boldsymbol{v})=((\curl\div)_h\boldsymbol{\tau}, \boldsymbol{v}) +(\grad\psi, \boldsymbol{v}).
\end{equation*}
Then the distributional mixed finite element method is to find $\boldsymbol{\sigma}_h\in {\Sigma}^{\rm tn}_h$, $\phi_h\in\mathring{\mathbb V}_h^{\grad}$ and $\boldsymbol{u}_h\in\mathring{\mathbb V}_h^{\curl}$ such that
\begin{subequations}
\begin{align}
\label{distribumfem1}	(\boldsymbol{\sigma}_h,\boldsymbol{\tau}_h)+b_h(\boldsymbol{\tau}_h,\psi_h;\boldsymbol{u}_h)&=0,  & & \boldsymbol{\tau}_h \in {\Sigma}^{\rm tn}_h, \psi_h\in\mathring{\mathbb V}_h^{\grad}, \\
\label{distribumfem2}	b_h(\boldsymbol{\sigma}_h,\phi_h;\boldsymbol{v}_h) &=-\langle\boldsymbol{f}, \boldsymbol{v}_h\rangle, & & \boldsymbol{v}_h \in \mathring{\mathbb V}_h^{\curl}.
\end{align}
\end{subequations}

We will derive two discrete inf-sup conditions for the linear form $b_h(\cdot,\cdot;\cdot)$.
To this end, introduce some mesh dependent norms.
For $\boldsymbol{\tau}\in\Sigma^{\rm tn}_h$, equip squared norm
\begin{align*}
\|\boldsymbol{\tau}\|_{H^{-1}((\curl\div)_h)}^2&:=\|\boldsymbol{\tau}\|^2+\|(\curl\div)_h\boldsymbol{\tau}\|_{H_h^{-1}(\div)}^2, 
\end{align*}
where 
$
\displaystyle\|\boldsymbol{v}\|_{H_h^{-1}(\div)}:=\sup_{\boldsymbol{w}_h\in\mathring{\mathbb V}_h^{\curl}}\frac{(\boldsymbol{v},\boldsymbol{w}_h)}{\|\boldsymbol{w}_h\|_{H(\curl)}}
$. The continuity of the bilinear form 
\begin{equation*}
b_h(\boldsymbol{\tau},\psi;\boldsymbol{v})\leq (\|\boldsymbol{\tau}\|_{H^{-1}((\curl\div)_h)}+|\psi|_1)\|\boldsymbol{v}\|_{H(\curl)},  \end{equation*}
for all $\boldsymbol{\tau}\in{\Sigma}^{\rm tn}, \psi\in\mathring{\mathbb V}_h^{\grad}, \boldsymbol{v}\in\mathring{\mathbb V}_h^{\curl}$, 
is straight forward by the definition of these norms.


\begin{lemma}\label{lm:discreteinfsup1}
For $\boldsymbol{v}_h \in \mathring{\mathbb V}_h^{\curl}$, it holds
\begin{equation}\label{eq:curldivdiscreteinfsup1}	
\|\boldsymbol{v}_h\|_{H(\curl)}\lesssim\sup_{\boldsymbol{\tau}_h \in {\Sigma}^{\rm tn}_h, \psi_h\in\mathring{\mathbb V}_h^{\grad}}\frac{b_h(\boldsymbol{\tau}_h,\psi_h;\boldsymbol{v}_h)}{\|\boldsymbol{\tau}_h\|_{H^{-1}((\curl\div)_h)}+|\psi_h|_1}.
\end{equation}
\end{lemma}
\begin{proof}
By Helmholtz decomposition \eqref{eq:helmholtz1}, given a $\bs v_h \in \mathring{\mathbb V}_h^{\curl}$, there exists $\bs u_h \in K_h^c$, $\tilde{\bs v}_h = (\curl\curl)_h\bs u_h$ and $\psi_h \in \mathring{\mathbb V}_{\ell+1, h}^{\grad}$ s.t.
\begin{equation}\label{eq:discretehemlholtz}
\bs v_h = (\curl\curl)_h\bs u_h \oplus^{L_2} \grad \psi_h = \tilde{\bs v}_h \oplus^{L_2} \grad \psi_h.
\end{equation} 
Then \begin{equation*}\curl \bs v_h = \curl \tilde{\bs v}_h, \quad \| \bs v_h \|^2= \| (\curl\curl)_h\bs u_h \|^2 + | \psi_h |_1^2. \end{equation*}

Set $\bs \tau_h=-I_h^{\rm tn}\mskw(\bs u_h + \tilde{\bs v}_h)\in \Sigma_{k-1,h}^{\rm tn}$.
By \eqref{eq:curdivhskwcurlcurlh}, 
\begin{equation*}
((\curl\div)_h\bs \tau_h, \bs v_h)=((\curl\curl)_h(\bs u_h + \tilde{\bs v}_h), \bs v_h) = \| (\curl\curl)_h\bs u_h\|^2 + \|\curl \bs v_h\|^2.
\end{equation*}
Consequently $b_h(\boldsymbol{\tau}_h,\psi_h;\boldsymbol{v}_h) = \|\bs v_h\|^2 + \|\curl \bs v_h\|^2$. 

It remains to control the norms. As the decomposition \eqref{eq:discretehemlholtz} is $L^2$-orthogonal, $|\psi_h|_1\leq \|\bs v_h\|$ and $\| (\curl\curl)_h\boldsymbol{u}_h\|\leq \|\bs v_h\|$. 
We control the negative norm by 
\begin{align*}
\| (\curl\div)_h \bs \tau_h\|_{H_h^{-1}(\div)}& = \sup_{\boldsymbol{w}_h\in\mathring{\mathbb V}_h^{\curl}}\frac{(\curl\boldsymbol{u}_h + \curl\tilde{\bs v}_h, \curl\boldsymbol{w}_h)}{\|\boldsymbol{w}_h\|_{H(\curl)}}\\
&\leq \|\curl\boldsymbol{u}_h\| + \|\curl \bs v_h\| \\
& \lesssim \|(\curl\curl)_h\bs u_h\| + \|\curl \bs v_h\| \leq \|\bs v_h\|_{H(\curl)},
\end{align*}
where we have used the discrete Poincar\'e	inequality
\eqref{eq:curlcurlhpoincare}.


By the scaling argument, the inverse inequality and the Poincar\'e	inequality
\eqref{eq:curlcurlhpoincare},
\begin{equation*}
\|\bs \tau_h \| \lesssim\|\bs u_h\| + \|\tilde{\bs v}_h\| \lesssim \|(\curl\curl)_h\bs u_h\| + \|\curl \bs v_h\| \leq 2\|\bs v_h\|_{H(\curl)},
\end{equation*}
as required.
%
\end{proof}


Introduce
\begin{equation*}
\|\boldsymbol{v}_h\|_{H((\grad\curl)_h)}^2 := \|\bs v_h\|^2 + \| (\grad\curl)_h \bs v_h\|^2. 
\end{equation*}
Again the continuity of the bilinear form in these norms
\begin{equation*}
b_h(\boldsymbol{\tau},\psi;\boldsymbol{v})\lesssim (\|\boldsymbol{\tau}\|+|\psi|_1)\|\boldsymbol{v}\|_{H((\grad\curl)_h)},  \boldsymbol{\tau}\in{\Sigma}^{\rm tn}_h, \psi\in\mathring{\mathbb V}_h^{\grad}, \boldsymbol{v}\in\mathring{\mathbb V}_h^{\curl}
\end{equation*}
is straight forward by the definition of these norms.

\begin{lemma}
For $\boldsymbol{v}_h \in \mathring{\mathbb V}_h^{\curl}$, it holds
\begin{equation}\label{eq:curldivdiscreteinfsup2}	
\|\boldsymbol{v}_h\|_{H((\grad\curl)_h)}\lesssim\sup_{\boldsymbol{\tau}_h \in {\Sigma}^{\rm tn}_h, \psi_h\in\mathring{\mathbb V}_h^{\grad}}\frac{b_h(\boldsymbol{\tau}_h,\psi_h;\boldsymbol{v}_h)}{\|\boldsymbol{\tau}_h\|+|\psi_h|_1}.
\end{equation}
\end{lemma}
\begin{proof}
We still use the Helmholtz decomposition \eqref{eq:discretehemlholtz} but choose \begin{equation*}\bs \tau_h = -(\grad \curl)_h \bs v_h =  -(\grad \curl)_h \tilde{\bs v}_h.\end{equation*}
Then 
\begin{equation*}
b_h(\boldsymbol{\tau}_h,\psi_h;\boldsymbol{v}_h) = \|  (\grad \curl)_h \bs v_h \|^2 + |\psi_h|_1^2. 
\end{equation*}
We end the proof by the estimates $\|\bs{\tau}_h\|+|\psi_h|_1\lesssim \|\boldsymbol{v}_h\|_{H((\grad\curl)_h)} $ and $ \|\boldsymbol{v}_h\| \lesssim \|(\grad \curl)_h \bs v_h \| + |\psi_h|_1$, in which we use the Poincar\'e inequality \eqref{eq:gradcurlhpoincare} for $\tilde{\bs v}_h$. 
\end{proof}

There are other variants of mesh-dependent norms. 
For $\bs\tau\in\Sigma^{\rm tn}$, equip a mesh-dependent squared norm
\begin{equation*}
\|\bs\tau\|_{0,h}^2:=\|\bs\tau\|^2 + \sum_{F\in\mathcal F_h}h_F\|\bs n\times\bs\tau\bs n\|_{F}^2.
\end{equation*}
By the inverse trace inequality, clearly we have $\|\bs\tau_h\|_{0,h}\eqsim \|\bs\tau_h\|$ for $\bs \tau_h\in \Sigma^{\rm tn}_h$.
For piecewise smooth vector-valued function $\bs v$, equip a mesh-dependent squared norm
\begin{equation*}
|\bs v|_{1,h}^2:= \sum_{T\in \mathcal T_h}\|\grad\bs v\|_{T}^2 + \sum_{F\in\mathcal F_h}h_F^{-1}\|\llbracket \bs v\rrbracket\|_{F}^2.
\end{equation*}
Then
\begin{equation*}
b_h(\boldsymbol{\tau},\psi;\boldsymbol{v})\lesssim \|\boldsymbol{\tau}\|_{0,h}|\curl\boldsymbol{v}|_{1,h}+|\psi|_1\|\bs v\|, \quad \boldsymbol{\tau}\in{\Sigma}^{\rm tn}, \psi\in H_0^1(\Omega), \boldsymbol{v}\in V_0^{\curl}.
\end{equation*}
One can prove the norm equivalence
\begin{equation}\label{eq:gradcurlhnormequiv}
\| (\grad\curl)_h \bs v_h\|\eqsim |\curl\bs v_h|_{1,h}, \quad \boldsymbol{v}_h \in \mathring{\mathbb V}_h^{\curl}
\end{equation}
and thus obtain the discrete inf-sup condition from \eqref{eq:curldivdiscreteinfsup2}
\begin{equation*}
\|\curl \boldsymbol{v}_h\|_{1,h}\lesssim\sup_{\boldsymbol{\tau}_h \in {\Sigma}^{\rm tn}_h, \psi\in\mathring{\mathbb V}_h^{\grad}}\frac{b_h(\boldsymbol{\tau}_h,\psi_h;\boldsymbol{v}_h)}{\|\boldsymbol{\tau}_h\|_{0,h}+|\psi_h|_1},
\end{equation*}
where $\|\curl \boldsymbol{v}_h\|_{1,h}^2:=\|\curl \boldsymbol{v}_h\|^2+|\curl \boldsymbol{v}_h|_{1,h}^2$.

The discrete coercivity on the null space is similar to Lemma \ref{lm:nullspacecoercivity}.
\begin{lemma}
For $\boldsymbol{\tau}_h\in{\Sigma}^{\rm tn}_h$ and $\psi_h\in\mathring{\mathbb V}_h^{\grad}$ satisfying 
\begin{equation*}
b_h(\boldsymbol{\tau}_h,\psi_h;\boldsymbol{v}_h)=0,\quad\boldsymbol{v}_h \in \mathring{\mathbb V}_h^{\curl},
\end{equation*}
it holds
\begin{equation}\label{eq:curldivcoercivitydiscrete}
\|\boldsymbol{\tau}_h\|_{H^{-1}((\curl\div)_h)}^2+|\psi_h|_1^2 = \|\boldsymbol{\tau}_h\|^2.
\end{equation}
\end{lemma}

Applying the Babu{\v{s}}ka-Brezzi theory~\cite{BoffiBrezziFortin2013}, from the discrete inf-sup conditions \eqref{eq:curldivdiscreteinfsup1} and \eqref{eq:curldivdiscreteinfsup2}, and the discrete coercivity \eqref{eq:curldivcoercivitydiscrete}, we achieve the well-posedness of the mixed finite element method \eqref{distribumfem1}-\eqref{distribumfem2}.
\begin{theorem}
\! The distributional mixed finite element method \eqref{distribumfem1}-\eqref{distribumfem2} for the quad-curl problem is well-posed.
We have the discrete stability results
\begin{align}
\notag
&\|{\boldsymbol{\sigma}}_h\|_{H^{-1}((\curl\div)_h)}+|{\phi}_h|_1+\|{\boldsymbol{u}}_h\|_{H(\curl)}\\
\label{eq:discretestability1}
&\quad \lesssim \sup_{\boldsymbol{\tau}_h\in{\Sigma}^{\rm tn}_h,\psi_h\in\mathring{\mathbb V}_h^{\grad},\boldsymbol{v}_h\in\mathring{\mathbb V}_h^{\curl}}\frac{A_h({\boldsymbol{\sigma}}_h,{\phi}_h,{\boldsymbol{u}}_h;\boldsymbol{\tau}_h,\psi_h,\boldsymbol{v}_h)}{\|\boldsymbol{\tau}_h\|_{H^{-1}((\curl\div)_h)}+|\psi_h|_1+\|\boldsymbol{v}_h\|_{H(\curl)}},
\end{align}
\begin{align}
\notag
&\|{\boldsymbol{\sigma}}_h\|+|{\phi}_h|_1+\|{\boldsymbol{u}}_h\|_{H((\grad\curl)_h)}\\
\label{eq:discretestability2}
&\quad \lesssim \sup_{\boldsymbol{\tau}_h\in{\Sigma}^{\rm tn}_h,\psi_h\in\mathring{\mathbb V}_h^{\grad},\boldsymbol{v}_h\in\mathring{\mathbb V}_h^{\curl}}\frac{A_h({\boldsymbol{\sigma}}_h,{\phi}_h,{\boldsymbol{u}}_h;\boldsymbol{\tau}_h,\psi_h,\boldsymbol{v}_h)}{\|\boldsymbol{\tau}_h\|+|\psi_h|_1+\|\boldsymbol{v}_h\|_{H((\grad\curl)_h)}},
\end{align}
for any ${\boldsymbol{\sigma}}_h\in{\Sigma}^{\rm tn}_h$, ${\phi}_h\in\mathring{\mathbb V}_h^{\grad}$ and ${\boldsymbol{u}}_h\in\mathring{\mathbb V}_h^{\curl}$, where
\begin{equation*}
A_h({\boldsymbol{\sigma}}_h,{\phi}_h,{\boldsymbol{u}}_h;\boldsymbol{\tau}_h,\psi_h,\boldsymbol{v}_h):=({\boldsymbol{\sigma}}_h,\boldsymbol{\tau}_h)+b_h(\boldsymbol{\tau}_h,\psi_h;{\boldsymbol{u}}_h)+b_h({\boldsymbol{\sigma}}_h,{\phi}_h;\boldsymbol{v}_h).
\end{equation*}
\end{theorem}

By choosing $\bs v_h=\grad\phi_h$ in \eqref{distribumfem2}, we get $\phi_h = 0$ from $\div\bs f=0$.

\subsection{Error analysis}

\begin{lemma}
Let $(\boldsymbol{\sigma},0,\boldsymbol{u})$ and $(\boldsymbol{\sigma}_h,0,\boldsymbol{u}_h)$ be the solution of the mixed formulation~\eqref{quadcurlmixed1}-\eqref{quadcurlmixed2} and the mixed finite element method \eqref{distribumfem1}-\eqref{distribumfem2} respectively. Assume $\boldsymbol{\sigma}\in\Sigma^{\rm tn}$, and $\boldsymbol{u},\curl\boldsymbol{u}\in H^1(\Omega;\mathbb R^3)$. Then
\begin{align}	
\notag
&\quad\; A_h(I_h^{\rm tn}\boldsymbol{\sigma}-\boldsymbol{\sigma}_h,0, I_h^{\curl}\boldsymbol{u}-\boldsymbol{u}_h;\boldsymbol{\tau}_h,\psi_h,\boldsymbol{v}_h) \\
\label{eq:error0}
&=(I_h^{\rm tn}\boldsymbol{\sigma}-\boldsymbol{\sigma},\boldsymbol{\tau}_h)+(I_h^{\curl}\boldsymbol{u}-\boldsymbol{u},\grad\psi_h)
\end{align}
holds for any $\boldsymbol{\tau}_h\in{\Sigma}^{\rm tn}_h$, $\psi_h\in\mathring{\mathbb V}_h^{\grad}$ and $\boldsymbol{v}_h\in\mathring{\mathbb V}_h^{\curl}$.
\end{lemma}
\begin{proof}
Subtract \eqref{distribumfem1}-\eqref{distribumfem2} from \eqref{quadcurlmixed1}-\eqref{quadcurlmixed2} and use \eqref{eq:curdivhcdprop} and \eqref{eq:curdivhIhcurl} to get error equations
\begin{align*}	
(\boldsymbol{\sigma}-\boldsymbol{\sigma}_h,\boldsymbol{\tau}_h)+b_h(\boldsymbol{\tau}_h,\psi_h; I_h^{\curl}\boldsymbol{u}-\boldsymbol{u}_h)&=(I_h^{\curl}\boldsymbol{u}-\boldsymbol{u},\grad\psi_h), 
\\
b_h(I_h^{\rm tn}\boldsymbol{\sigma}-\boldsymbol{\sigma}_h,0;\boldsymbol{v}_h)&=0. 
\end{align*}
Then subtract $(\boldsymbol{\sigma}-I_h^{\rm tn}\boldsymbol{\sigma},\boldsymbol{\tau}_h)$ to get \eqref{eq:error0}.
\end{proof}

\begin{theorem}
Let $(\boldsymbol{\sigma},0,\boldsymbol{u})$ and $(\boldsymbol{\sigma}_h,0,\boldsymbol{u}_h)$ be the solution of the mixed formulation~\eqref{quadcurlmixed1}-\eqref{quadcurlmixed2} and the mixed finite element method \eqref{distribumfem1}-\eqref{distribumfem2} respectively. Assume $\boldsymbol{\sigma}\in H^k(\Omega;\mathbb T)$ and $\boldsymbol{u}, \curl\boldsymbol{u}\in H^k(\Omega;\mathbb R^3)$. Then
\begin{align}
\label{eq:errorestimate1}
\|\boldsymbol{\sigma}-\boldsymbol{\sigma}_h\|+\|I_h^{\curl}\boldsymbol{u}-\boldsymbol{u}_h\|_{H((\grad\curl)_h)}&\lesssim h^k(|\boldsymbol{\sigma}|_k+|\boldsymbol{u}|_k), \\
\label{eq:errorestimate2}
\|\boldsymbol{u}-\boldsymbol{u}_h\|_{H(\curl)}+h|\curl(\boldsymbol{u}-\boldsymbol{u}_h)|_{1,h}&\lesssim h^k(|\boldsymbol{\sigma}|_k+|\boldsymbol{u}|_k+|\curl\boldsymbol{u}|_k).
\end{align}
\end{theorem}
\begin{proof}
It follows from	the stability results \eqref{eq:discretestability1}-\eqref{eq:discretestability2} and \eqref{eq:error0} that
\begin{align*}
&\quad\|I_h^{\rm tn}\boldsymbol{\sigma}-\boldsymbol{\sigma}_h\|+\|I_h^{\curl}\boldsymbol{u}-\boldsymbol{u}_h\|_{H(\curl)}+\|I_h^{\curl}\boldsymbol{u}-\boldsymbol{u}_h\|_{H((\grad\curl)_h)} \\
&\lesssim \|\boldsymbol{\sigma}-I_h^{\rm tn}\boldsymbol{\sigma}\|+\|\boldsymbol{u}-I_h^{\curl}\boldsymbol{u}\|.
\end{align*}
Hence \eqref{eq:errorestimate1}-\eqref{eq:errorestimate2} follow from the triangle inequality, the norm equivalence \eqref{eq:gradcurlhnormequiv} and interpolation error estimates.
\end{proof}

Using the usual duality argument, we can additionally derive the superconvergence of $\|\curl(I_h^{\curl}\boldsymbol{u}-\boldsymbol{u}_h)\|_{0}$. 
Let $\widetilde{\boldsymbol{u}}$ be the solution of the
auxiliary problem:
\begin{equation}\label{bigradcurlproblemdual}
	\begin{aligned}
		&\left\{\begin{aligned}
			-\curl\Delta\curl\widetilde{\boldsymbol{u}}&=-\curl^2(I_h^{\curl}\boldsymbol{u}-\boldsymbol{u}_h) & & \text { in } \Omega, \\
			\operatorname{div}\widetilde{\boldsymbol{u}} &=0 & & \text { in } \Omega, \\
			\widetilde{\boldsymbol{u}} \times \boldsymbol{n}=\operatorname{curl}\widetilde{\boldsymbol{u}}\times \boldsymbol{n} &=0 & & \text { on } \partial \Omega.
		\end{aligned}\right.
	\end{aligned}
\end{equation}
We assume that
$\widetilde{\boldsymbol{u}}\in H^1(\Omega;\mathbb R^3)\cap H_0(\curl,\Omega)$ with
the bound
\begin{equation}
\label{bigradcurldualregularity}
\|\widetilde{\boldsymbol{\sigma}}\|_{1}+\|\widetilde{\boldsymbol{u}}\|_{1}\lesssim \|\curl(I_h^{\curl}\boldsymbol{u}-\boldsymbol{u}_h)\|,
\end{equation}
where $\widetilde{\boldsymbol{\sigma}}:=\grad\curl\widetilde{\boldsymbol{u}}$.
When $\Omega$ is convex, regularity \eqref{bigradcurldualregularity} has been obtained in~\cite[Lemma A.1]{Huang2020}.

\begin{theorem}
Let $(\boldsymbol{\sigma},0,\boldsymbol{u})$ and $(\boldsymbol{\sigma}_h,0,\boldsymbol{u}_h)$ be the solution of the mixed formulation~\eqref{quadcurlmixed1}-\eqref{quadcurlmixed2} and the mixed finite element method \eqref{distribumfem1}-\eqref{distribumfem2} respectively.
Assume the regularity condition \eqref{bigradcurldualregularity} holds, $\bs f\in L^2(\Omega;\mathbb R^3)$, $\boldsymbol{\sigma}\in H^k(\Omega;\mathbb T)$ and $\boldsymbol{u}, \curl\boldsymbol{u}\in H^k(\Omega;\mathbb R^3)$.
Then
\begin{equation}\label{eq:curlsupererrorestimate}
\|\curl(I_h^{\curl}\boldsymbol{u}-\boldsymbol{u}_h)\|\lesssim h^{k+1}(|\boldsymbol{\sigma}|_k+|\boldsymbol{u}|_k+\delta_{k1}\|\boldsymbol{f}\|),
\end{equation}
where $\delta_{11}=1$ and $\delta_{k1}=0$ for $k\neq1$.
\end{theorem}
\begin{proof}
By the first equation of \eqref{bigradcurlproblemdual}, we have
$$
\|\curl(I_h^{\curl}\boldsymbol{u}-\boldsymbol{u}_h)\|^2=(\div\widetilde{\boldsymbol{\sigma}}, \curl(I_h^{\curl}\boldsymbol{u}-\boldsymbol{u}_h))=b_h(\widetilde{\boldsymbol{\sigma}}, 0; I_h^{\curl}\boldsymbol{u}-\boldsymbol{u}_h).
$$
Then we get from \eqref{eq:curdivhIhcurl} and \eqref{eq:curdivhcdprop} that
$$
\|\curl(I_h^{\curl}\boldsymbol{u}-\boldsymbol{u}_h)\|^2=b_h(I_h^{\rm tn}\widetilde{\boldsymbol{\sigma}}, 0; I_h^{\curl}\boldsymbol{u}-\boldsymbol{u}_h)=b_h(I_h^{\rm tn}\widetilde{\boldsymbol{\sigma}}, 0; \boldsymbol{u}-\boldsymbol{u}_h),
$$
which together with \eqref{quadcurlmixed1} and \eqref{distribumfem1} gives
$$
\|\curl(I_h^{\curl}\boldsymbol{u}-\boldsymbol{u}_h)\|^2= - (\boldsymbol{\sigma}-\boldsymbol{\sigma}_h, I_h^{\rm tn}\widetilde{\boldsymbol{\sigma}}).
$$
On the other side, it follows from \eqref{quadcurlmixed2}, \eqref{distribumfem2} and \eqref{eq:curdivhIhcurl} that
\begin{align*}
- (\boldsymbol{\sigma}-\boldsymbol{\sigma}_h, \widetilde{\boldsymbol{\sigma}})&=- (\boldsymbol{\sigma}-\boldsymbol{\sigma}_h, \grad\curl\widetilde{\boldsymbol{u}})=b_h(\boldsymbol{\sigma}-\boldsymbol{\sigma}_h, 0; \widetilde{\boldsymbol{u}}) \\
&=b_h(\boldsymbol{\sigma}-\boldsymbol{\sigma}_h, 0; \widetilde{\boldsymbol{u}}-I_h^{\curl}\widetilde{\boldsymbol{u}})=b_h(\boldsymbol{\sigma}, 0; \widetilde{\boldsymbol{u}}-I_h^{\curl}\widetilde{\boldsymbol{u}}).
\end{align*}
Combining the last two equations and the estimate of $I_h^{\rm tn}$ yields
\begin{align}
\|\curl(I_h^{\curl}\boldsymbol{u}-\boldsymbol{u}_h)\|^2&=(\boldsymbol{\sigma}-\boldsymbol{\sigma}_h, \widetilde{\boldsymbol{\sigma}}-I_h^{\rm tn}\widetilde{\boldsymbol{\sigma}})+b_h(\boldsymbol{\sigma}, 0; \widetilde{\boldsymbol{u}}-I_h^{\curl}\widetilde{\boldsymbol{u}})  \notag \\
&\lesssim h\|\boldsymbol{\sigma}-\boldsymbol{\sigma}_h\|\|\widetilde{\boldsymbol{\sigma}}\|_1 +b_h(\boldsymbol{\sigma}, 0; \widetilde{\boldsymbol{u}}-I_h^{\curl}\widetilde{\boldsymbol{u}}). \label{eq:202211221}
\end{align}

When $k\geq 2$, applying \eqref{eq:curdivhIhcurl} again, we get from \eqref{eq:202211221} that 
\begin{align*}
\|\curl(I_h^{\curl}\boldsymbol{u}-\boldsymbol{u}_h)\|^2&
\lesssim h\|\boldsymbol{\sigma}-\boldsymbol{\sigma}_h\|\|\widetilde{\boldsymbol{\sigma}}\|_1+b_h(\boldsymbol{\sigma}-I_h^{\rm tn}\boldsymbol{\sigma}, 0; \widetilde{\boldsymbol{u}}-I_h^{\curl}\widetilde{\boldsymbol{u}}) \\
&\leq h\|\boldsymbol{\sigma}-\boldsymbol{\sigma}_h\|\|\widetilde{\boldsymbol{\sigma}}\|_1 \\
&\quad + \|\boldsymbol{\sigma}-I_h^{\rm tn}\boldsymbol{\sigma}\|_{0,h}|\curl(\widetilde{\boldsymbol{u}}-I_h^{\curl}\widetilde{\boldsymbol{u}})|_{1,h} \\
&\lesssim h(\|\boldsymbol{\sigma}-\boldsymbol{\sigma}_h\|+\|\boldsymbol{\sigma}-I_h^{\rm tn}\boldsymbol{\sigma}\|_{0,h})\|\widetilde{\boldsymbol{\sigma}}\|_1.
\end{align*}
Therefore \eqref{eq:curlsupererrorestimate} follows from \eqref{eq:errorestimate1}, interpolation error estimate of $I_h^{\rm tn}$ and \eqref{bigradcurldualregularity}.

Next consider case $k=1$.
By the de Rham complex~\cite{CostabelMcIntosh2010}, there exists $\boldsymbol{w}\in H_0^1(\Omega;\mathbb R^3)$ such that
$$
\curl\boldsymbol{w}=\curl(\widetilde{\boldsymbol{u}}-I_h^{\curl}\widetilde{\boldsymbol{u}}), \quad \|\boldsymbol{w}\|_1\lesssim \|\curl\boldsymbol{w}\|\lesssim h|\curl\widetilde{\boldsymbol{u}}|_1.
$$
Then apply the commutative property of $I_h^{\curl}$ to get $\curl(I_h^{\curl}\boldsymbol{w})=0$. Employing the second equation of problem~\eqref{quadcurl2ndsystem} and the estimate of $I_h^{\curl}$, we acquire
\begin{align*}	
b_h(\boldsymbol{\sigma}, 0; \widetilde{\boldsymbol{u}}-I_h^{\curl}\widetilde{\boldsymbol{u}})&=(\operatorname{div} \boldsymbol{\sigma}, \curl(\widetilde{\boldsymbol{u}}-I_h^{\curl}\widetilde{\boldsymbol{u}}))=(\operatorname{div} \boldsymbol{\sigma}, \curl({\boldsymbol{w}}-I_h^{\curl}{\boldsymbol{w}})) \\
&= (\curl\operatorname{div} \boldsymbol{\sigma}, {\boldsymbol{w}}-I_h^{\curl}{\boldsymbol{w}})= -(\boldsymbol{f}, {\boldsymbol{w}}-I_h^{\curl}{\boldsymbol{w}}) \\
&\lesssim h\|\boldsymbol{f}\|\|\boldsymbol{w}\|_1 \lesssim h^2\|\boldsymbol{f}\||\curl\widetilde{\boldsymbol{u}}|_1.
\end{align*}
Finally, \eqref{eq:curlsupererrorestimate} follows from \eqref{eq:202211221}, \eqref{eq:errorestimate1} and \eqref{bigradcurldualregularity}.
\end{proof}


We perform numerical experiments to support the theoretical results of the distributional mixed method~\eqref{distribumfem1}-\eqref{distribumfem2}.  Let $\Omega=(0,1)^3$. Choose
the function $\boldsymbol{f}$ in~(\ref{bigradcurlproblem}) such that the exact solution of (\ref{bigradcurlproblem}) is
\[
\boldsymbol{u} = \curl\begin{pmatrix}
\sin^3(\pi x)\sin^3(\pi y)\sin^3(\pi z) \\
\sin^3(\pi x)\sin^3(\pi y)\sin^3(\pi z) \\
0
\end{pmatrix},
\]
and let $\boldsymbol{\sigma}:= \grad\curl\boldsymbol{u}$. 
To break the symmetry, the initial unstructured mesh of $\Omega$ is shown in Fig.~\ref{fig:initmesh}, which is a perturbation of the uniform mesh with $h=1/2$.
Then we take uniform refinement of this initial unstructured triangulation.
\begin{figure}[htbp]
\begin{center}
\includegraphics[width=5cm]{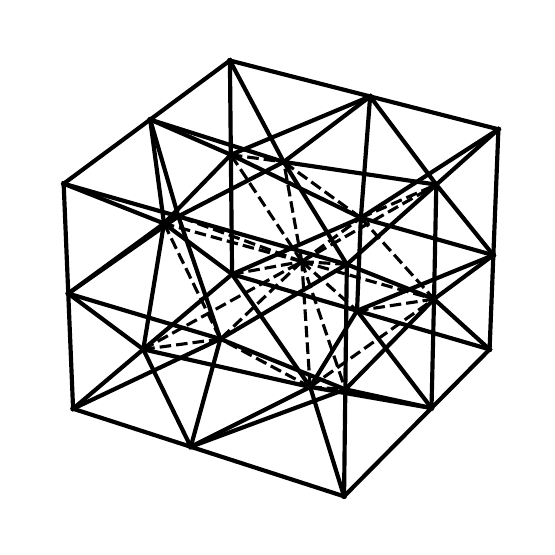}
\caption{An initial perturbed mesh of the uniform mesh with $h=1/2$.}
\label{fig:initmesh}
\end{center}
\end{figure}

We will implement the hybridized version of~\eqref{distribumfem1}-\eqref{distribumfem2}; see~\eqref{hybridishhjtype1}-\eqref{hybridishhjtype2}. Therefore the number of DoFs is $\dim \mathring{\mathbb{V}}^{\curl}_{(k,\ell), h} + \dim\Lambda_{k-1,h}$, where $\Lambda_{k-1,h}$ is the space of Lagrange multiplier. Numerical results of $\|\boldsymbol{\sigma}-\boldsymbol{\sigma}_h\|$, $\|{\boldsymbol{u}}-\boldsymbol{u}_{h}\|$,  $\|\curl({\boldsymbol{u}}-\boldsymbol{u}_{h})\|$ and $\|\grad_{\mathcal T_h}\curl(\boldsymbol{u}-\boldsymbol{u}_h)\|$ of the  distributional mixed method \eqref{distribumfem1}-\eqref{distribumfem2} with $k=1$ and $\ell=0,1$ are shown in Table~\ref{table:k1l01}.  It is observed that $\|\boldsymbol{\sigma}-\boldsymbol{\sigma}_h\|=\mathcal{O}(h)$, and $\|\curl({\boldsymbol{u}}-\boldsymbol{u}_{h})\| = \mathcal{O}(h)$ numerically,   which coincide with the theoretical
error estimates in (\ref{eq:errorestimate1}) and (\ref{eq:errorestimate2}). The error $\|{\boldsymbol{u}}-\boldsymbol{u}_{h}\| = \mathcal{O}(h^{\ell+1})$ is also observed but not included in (\ref{eq:errorestimate1}) and (\ref{eq:errorestimate2}). 
\begin{table}[ht]
		\setlength{\abovecaptionskip}{0.cm}
		\setlength{\belowcaptionskip}{-0.cm}
	\centering
	\caption{Errors for the distributional mixed finite element method~\eqref{distribumfem1}-\eqref{distribumfem2}~with $k=1$.}
	\label{table:k1l01}
	\scalebox{0.91}{
	\renewcommand{\arraystretch}{1.25}		
\begin{tabular}[t]{ccccc}
    \toprule
    $h$ &$2^{-2}$ &$2^{-3}$ &$2^{-4}$ &$2^{-5}$ \\
    \midrule
    \# DoFs for $\ell=0$ & 2,332 & 17,240 & 132,400 & 1,037,408 \\
    \# DoFs for $\ell=1$ & 2,936 & 21,424 & 163,424 & 1,276,096 \\
    $\|\boldsymbol{\sigma}-\boldsymbol{\sigma}_h\|$ &
    1.51E+02 & 8.62E+01 & 4.48E+01 & 2.26E+01 \\
    order &$-$ & 0.81 & 0.95 & 0.99 \\
      $\|\curl(\boldsymbol{u}-\boldsymbol{u}_h)\|$ &   
    1.48E+01 & 5.60E+00 & 2.39E+00 & 1.13E+00 \\
    order &$-$ & 1.40 & 1.23 & 1.08 \\
    $\|\grad_{\mathcal T_h}\curl(\boldsymbol{u}-\boldsymbol{u}_h)\|$ &
    1.46E+02 & 1.46E+02 & 1.46E+02 & 1.46E+02 \\
    order &$-$ & 0 & 0 & 0 \\
    $\|\boldsymbol{u}-\boldsymbol{u}_h\|$ for $\ell=0$
    & 1.49E+00 & 5.23E-01 & 2.19E-01 & 1.03E-01 \\
    order &$-$ & 1.51 & 1.25 & 1.09 \\
  $\|\boldsymbol{u}-\boldsymbol{u}_h\|$ for $\ell=1$ & 
    1.04E+00 & 2.82E-01 & 7.32E-02 & 1.85E-02 \\
    order &$-$ & 1.88 & 1.95 & 1.98 \\
    \bottomrule
\end{tabular}
	}
\end{table}

\subsection{Post-processing}\label{subsec:postprocess}
Post-processing can be also applied to improve the approximation. It follows from the standard procedure of the stable mixed methods and will be briefly reviewed below. 

We will construct a new superconvergent approximation to $\boldsymbol{u}$ in virtue of the optimal result of $\boldsymbol{\sigma}$ in \eqref{eq:errorestimate1}. 
%
Introduce discrete space 
\begin{align*}
\mathbb V_h^{\ast}&:=\{\boldsymbol{v}_h\in L^2(\Omega;\mathbb R^3): \boldsymbol{v}_h|_T \in\mathbb P_{k}(T;\mathbb R^3)+\boldsymbol{x}\times\mathbb P_{k}(T;\mathbb R^3) \textrm{ for } T \in \mathcal{T}_{h}\}.
\end{align*}
Let $\mathring{\mathbb P}_{k+1}(T)$ be the subspace of $\mathbb P_{k+1}(T)$ with vanishing funcition values at vertices of $T$.
For each $T\in\mathcal{T}_h$, define the new approximation $\boldsymbol{u}_h^{\ast}\in \mathbb V_h^{\ast}$ to $\boldsymbol{u}$ piecewisely as a solution of the following problem: 
\begin{align}
(\boldsymbol{u}_h^{\ast}\cdot\boldsymbol t, q)_e &=(\boldsymbol u_h\cdot\boldsymbol t, q)_e \quad\;\forall~q\in\mathbb P_{0}(e),  e\in\mathcal E(T), \label{eq:postprocess1}\\
(\grad\curl\boldsymbol{u}_h^{\ast}, \boldsymbol{q})_T&=(\boldsymbol{\sigma}_h, \boldsymbol{q})_T\qquad\forall~\boldsymbol{q}\in\grad\curl\mathbb P_{k+1}(T;\mathbb R^3), \label{eq:postprocess2}\\
(\boldsymbol{u}_h^{\ast}, \boldsymbol{q})_T&=(\boldsymbol{u}_h, \boldsymbol{q})_T\qquad\forall~\boldsymbol{q}\in\grad\mathring{\mathbb P}_{k+1}(T). \label{eq:postprocess3}
\end{align}
It is easy to verify that the local problem \eqref{eq:postprocess1}-\eqref{eq:postprocess3} is well-posed.


\begin{theorem}\label{thm:uuhstar1}
Let $(\boldsymbol{\sigma},0,\boldsymbol{u})$ and $(\boldsymbol{\sigma}_h,0,\boldsymbol{u}_h)$ be the solution of the mixed formulation~\eqref{quadcurlmixed1}-\eqref{quadcurlmixed2} and the mixed finite element method \eqref{distribumfem1}-\eqref{distribumfem2} respectively. Assume $\boldsymbol{\sigma}\in H^k(\Omega;\mathbb T)$ and $\boldsymbol{u}, \curl\boldsymbol{u}\in H^{k+1}(\Omega;\mathbb R^3)$. Then
\begin{equation}
\label{uhast:errorestimate1}
\|\grad_{\mathcal T_h}\curl_{\mathcal T_h}(\boldsymbol{u}-\boldsymbol{u}_h^{\ast})\|\lesssim h^k(|\boldsymbol{\sigma}|_k+|\boldsymbol{u}|_k+|\curl\boldsymbol{u}|_{k+1}),
\end{equation}
Further assume the regularity condition \eqref{bigradcurldualregularity} holds and $\bs f\in L^2(\Omega;\mathbb R^3)$. Then
\begin{equation}
\label{uhast:errorestimate2}
\|\curl_{\mathcal T_h}(\boldsymbol{u}-\boldsymbol{u}_h^{\ast})\|\lesssim h^{k+1}(|\boldsymbol{\sigma}|_k+|\boldsymbol{u}|_k+|\curl\boldsymbol{u}|_{k+1}+\delta_{k1}\|\boldsymbol{f}\|).
\end{equation}
\end{theorem}
\begin{proof}
Denote $\boldsymbol{w}=(I-I_{T,1}^{\curl})(I_{T, k+1}^{\curl}\boldsymbol{u}-\boldsymbol{u}_h^{\ast})\in\mathbb P_{k}(T;\mathbb R^3)+\boldsymbol{x}\times\mathbb P_{k}(T;\mathbb R^3)$. From the first equation of problem~\eqref{quadcurl2ndsystem} and \eqref{eq:postprocess2} with $\boldsymbol{q}=\grad\curl\boldsymbol{w}$, it follows that
\begin{align*}	
|\curl\boldsymbol{w}|_{1,T}^2&=(\grad\curl(I_{T, k+1}^{\curl}\boldsymbol{u}-\boldsymbol{u}_h^{\ast}), \grad\curl\boldsymbol{w})_T \\
&=(\grad\curl(I_{T, k+1}^{\curl}\boldsymbol{u}-\boldsymbol{u}), \grad\curl\boldsymbol{w})_T +(\boldsymbol{\sigma}-\boldsymbol{\sigma}_h, \grad\curl\boldsymbol{w})_T,
\end{align*}
which means
\begin{equation}\label{eq:20221122}	
|\curl\boldsymbol{w}|_{1,T}\leq |\curl(I_{T, k+1}^{\curl}\boldsymbol{u}-\boldsymbol{u})|_{1,T} + \|\boldsymbol{\sigma}-\boldsymbol{\sigma}_h\|_{0,T}.
\end{equation}
Then
$$
|\curl(\boldsymbol{u}-\boldsymbol{u}_h^{\ast})|_{1,T}\leq 2|\curl(I_{T, k+1}^{\curl}\boldsymbol{u}-\boldsymbol{u})|_{1,T} + \|\boldsymbol{\sigma}-\boldsymbol{\sigma}_h\|_{0,T}.
$$
Hence \eqref{uhast:errorestimate1} follows from \eqref{eq:errorestimate1} and the error estimate of $I_{T, k+1}^{\curl}$.

Next we derive \eqref{uhast:errorestimate2}.
It is easy to see that $I_{T,1}^{\curl}\boldsymbol{w}=0$. By the error estimate of $I_{T,1}^{\curl}$,
\begin{equation*}
\|\curl\boldsymbol{w}\|_{0,T}=\|\curl(\boldsymbol{w}-I_{T,1}^{\curl}\boldsymbol{w})\|_{0,T}\lesssim h_K|\curl\boldsymbol{w}|_{1,T},
\end{equation*}
which together with \eqref{eq:20221122} gives
\[
\|\curl\boldsymbol{w}\|_{0,T}\lesssim h_K|\curl\boldsymbol{w}|_{1,T}\lesssim h_T|\curl(I_{T, k+1}^{\curl}\boldsymbol{u}-\boldsymbol{u})|_{1,T} + h_T\|\boldsymbol{\sigma}-\boldsymbol{\sigma}_h\|_{0,T}.
\]
On the other side, by \eqref{eq:postprocess1}, the commutative property of $I_{T,1}^{\curl}$ and the norm equivalence, we have
\begin{align*}
\|\curl\big(I_{T,1}^{\curl}(I_{T, k+1}^{\curl}\boldsymbol{u}-\boldsymbol{u}_h^{\ast})\big)\|_{0,T}&=\|\curl\big(I_{T,1}^{\curl}(I_h^{\curl}\boldsymbol{u}-\boldsymbol{u}_h))\|_{0,T} \\
&\lesssim \|\curl(I_h^{\curl}\boldsymbol{u}-\boldsymbol{u}_h)\|_{0,T}. \notag
\end{align*}
Then we get from the last two inequalities that
\begin{align*}
\|\curl(\boldsymbol{u}-\boldsymbol{u}_h^{\ast})\|_{0,T}& \leq \|\curl(\boldsymbol{u}-I_{T, k+1}^{\curl}\boldsymbol{u})\|_{0,T}+\|\curl\boldsymbol{w}\|_{0,T}\\
&\quad + \|\curl(I_{T,1}^{\curl}(I_{T, k+1}^{\curl}\boldsymbol{u}-\boldsymbol{u}_h^{\ast}))\|_{0,T} \\
& \lesssim \|\curl(\boldsymbol{u}-I_{T, k+1}^{\curl}\boldsymbol{u})\|_{0,T} + h_T|\curl(I_{T, k+1}^{\curl}\boldsymbol{u}-\boldsymbol{u})|_{1,T} \\
&\quad + h_T\|\boldsymbol{\sigma}-\boldsymbol{\sigma}_h\|_{0,T} + \|\curl(I_h^{\curl}\boldsymbol{u}-\boldsymbol{u}_h)\|_{0,T}.
\end{align*}
Finally, estimate \eqref{uhast:errorestimate2} follows from the error estimate of $I_{T, k+1}^{\curl}$, \eqref{eq:errorestimate1} and \eqref{eq:curlsupererrorestimate}.
\end{proof}

\begin{table}[htp]
		\setlength{\abovecaptionskip}{0.cm}
		\setlength{\belowcaptionskip}{-0.cm}
	\centering
\caption{Errors for the post-processing for $k=1$.}\label{table:k1l02}
	\scalebox{0.91}{
	\renewcommand{\arraystretch}{1.25}		
\begin{tabular}[t]{ccccc}
    \toprule
    $h$ &$2^{-2}$ &$2^{-3}$ &$2^{-4}$ &$2^{-5}$ \\
    \midrule
    \# DoFs for $\ell=0$ & 2,332 & 17,240 & 132,400 & 1,037,408 \\
    \# DoFs for $\ell=1$ & 2,936 & 21,424 & 163,424 & 1,276,096 \\
      $\|\curl_{\mathcal T_h}(\boldsymbol{u}-\boldsymbol{u}_h^{\ast})\|$ &   
    9.25E+00 & 2.71E+00 & 7.18E-01 & 1.83E-01 \\
    order &$-$ & 1.77 & 1.92 & 1.98 \\
    $\|\grad_{\mathcal T_h}\curl_{\mathcal T_h}(\boldsymbol{u}-\boldsymbol{u}_h^{\ast})\|$ &
    1.52E+02 & 8.62E+01 & 4.48E+01 & 2.26E+01 \\
    order &$-$ & 0.81 & 0.95 & 0.99 \\
    \bottomrule
\end{tabular}
	}
\end{table}

Numerical results of the post-processing $\boldsymbol{u}_h^{\ast}$ for $k=1$ and $\ell=0,1$ are listed in Table~\ref{table:k1l02}. 
We can observe that $\|\curl_{\mathcal T_h}(\boldsymbol{u}-\boldsymbol{u}_h^{\ast})\|= \mathcal{O}(h^2)$ and $\|\grad_{\mathcal T_h}\curl_{\mathcal T_h}(\boldsymbol{u}-\boldsymbol{u}_h^{\ast})\|= \mathcal{O}(h)$ numerically, which are one order higher than $\|\curl({\boldsymbol{u}}-\boldsymbol{u}_{h})\|$ and $\|\grad_{\mathcal T_h}\curl({\boldsymbol{u}}-\boldsymbol{u}_{h})\|$ respectively.
Indeed, $\|\grad_{\mathcal T_h}\curl({\boldsymbol{u}}-\boldsymbol{u}_{h})\|=\|\boldsymbol{\sigma}-\boldsymbol{\sigma}_h\|$.

\section{Hybridization of distributional mixed finite element method}\label{sec:hybridization}
In this section we will hybridize the distributional mixed finite element method \eqref{distribumfem1}-\eqref{distribumfem2} following the framework in~\cite{ArnoldBrezzi1985}. We introduce a Lagrange multiplier for the tangential-normal continuity of $\boldsymbol{\sigma}$ which can be treat as an approximation of $(\boldsymbol{n}_F\times \curl\boldsymbol{u})|_{\mathcal F_h}$.  As $\bs \sigma$ is discontinuous, it can be eliminated element-wise and the size of the resulting linear system  is reduced, which is easier to solve than the saddle point system obtained by the mixed method \eqref{distribumfem1}-\eqref{distribumfem2}.  
With the hybridized method, we can also establish the equivalence of the mixed method \eqref{distribumfem1}-\eqref{distribumfem2} to a weak Galerkin method without stabilization and nonconforming finite element methods in~\cite{Huang2020,ZhengHuXu2011}.

To this end, introduce two finite element spaces
\begin{align*}
	{\Sigma}_{k-1,h}^{-1}:=&\{\boldsymbol{\tau}_h \in L^2(\Omega ; \mathbb{T}): \boldsymbol{\tau}_h|_T \in \mathbb{P}_{k-1}(T ; \mathbb{T}) \text{ for each } T \in \mathcal{T}_h\},\\
	\Lambda_{k-1,h}:=&\{\boldsymbol{\mu}_h \in L^2(\mathcal{F}_h;\mathbb{R}^{3}):\boldsymbol{\mu}_h|_F \in \mathbb{P}_{k-1}(F; \mathbb{R}^{3}) \text{ and }  \boldsymbol{\mu}_h\cdot\boldsymbol{n}|_F=0 \text{ for each }F \in \mathring{\mathcal{F}}_h, \\
         &\qquad\qquad\qquad\qquad\;\text{ and } \boldsymbol{\mu}_h=0\text{ on } \mathcal{F}_h \backslash \mathring{\mathcal{F}}_h\}.
\end{align*}
The space $\Lambda_h$ is introduced as the Lagrange multiplier to impose the tangential-normal continuity and is equipped with squared norm
\begin{equation*}
\left\|\boldsymbol{\mu}_h\right\|_{\alpha, h}^2:=\sum_{T \in \mathcal{T}_h} \sum_{F \in \mathcal{F}(T)} h_F^{-2 \alpha}\left\|\boldsymbol{\mu}_h\right\|_{F}^2, \quad \alpha=\pm 1 / 2.
\end{equation*}

\subsection{Hybridization}
The hybridization of the mixed finite element method \eqref{distribumfem1}-\eqref{distribumfem2} is to find
$(\boldsymbol{\sigma}_h, \boldsymbol{u}_h, \phi_h, \boldsymbol{\lambda}_h) \in {\Sigma}_{k-1,h}^{-1} \times \mathring{\mathbb{V}}^{\curl}_{(k,\ell), h} \times \mathring{\mathbb{V}}^{\grad}_{\ell+1, h} \times \Lambda_{k-1,h}$
such that
\begin{subequations}\label{eq:hy}
 \begin{align}
	\label{hybridishhjtype1}	
(\boldsymbol{\sigma}_{h}, \boldsymbol{\tau}_{h})+b_h(\boldsymbol{\tau}_{h},\psi_h; \boldsymbol{u}_h) \!+ c_h(\boldsymbol{\tau}_{h}, \boldsymbol{\lambda}_h) &=0, \quad\;\;\;\;\;\;\;\boldsymbol{\tau}_{h} \in {\Sigma}_{k-1,h}^{-1}, \psi_h\in\mathring{\mathbb{V}}^{\grad}_{\ell+1,h}, \\
\label{hybridishhjtype2}	
b_h(\boldsymbol{\sigma}_{h},\phi_h; \boldsymbol{v}_h) \!+ c_h(\boldsymbol{\sigma}_{h}, \boldsymbol{\mu}_h) &=\!\!-\langle\boldsymbol{f}, \boldsymbol{v}_h\rangle, \boldsymbol{v}_{h} \in \mathring{\mathbb{V}}^{\curl}_{(k,\ell),h}, \boldsymbol{\mu}_h\in \Lambda_{k-1,h}, 
\end{align}
\end{subequations}
where the bilinear form
$
c_h(\boldsymbol{\tau}_{h}, \boldsymbol{\lambda}_h):=-\sum_{T \in \mathcal{T}_{h}}\left(\boldsymbol{n}\times \boldsymbol{\tau}_h\boldsymbol{n}_F, \boldsymbol{\lambda}_h\right)_{\partial T}
$ is introduced to impose the tangential-normal continuity. 

\begin{lemma}
There holds the following inf-sup condition
\begin{align}\begin{split}\label{hybridizationinfsup}
	&	\|\boldsymbol{v}_h\|_{H((\grad\curl)_h)}+\left\|\boldsymbol{n}_F\times \curl\boldsymbol{v}_h-\boldsymbol{\mu}_h\right\|_{1 / 2, h}\\
	\lesssim& \sup _{\boldsymbol{\tau}_h \in\Sigma_{k-1,h}^{-1},\psi_h\in\mathring{\mathbb{V}}_{\ell+1, h}^{\grad}} \frac{b_h\left(\boldsymbol{\tau}_h ,\psi_h; \boldsymbol{v}_h\right) + c_h(\boldsymbol{\tau}_h,\boldsymbol{\mu}_h)}{\left\|\boldsymbol{\tau}_h\right\|+|\psi_h|_1}, \ \boldsymbol{v}_h \in \mathring{\mathbb{V}}_{(k,\ell),h}^{\curl}, \boldsymbol{\mu}_h \in \Lambda_{k-1,h}.
\end{split}\end{align}
\end{lemma}
\begin{proof}
Let $\boldsymbol{\tau}_h\in\Sigma_{k-1,h}^{-1}$ be determined as follows: for $T\in\mathcal T_h$, 
\begin{align*}
(\boldsymbol{n}\times\boldsymbol{\tau}_h\boldsymbol{n}_F)|_F&=\frac{1}{h_F}(\boldsymbol{n}_F\times\curl \boldsymbol{v}_h-\boldsymbol{\mu}_h),\qquad F\in\mathcal F(T),
\\
(\boldsymbol{\tau}_h, \boldsymbol{q})_T&=-(\grad\curl\boldsymbol{v}_h, \boldsymbol{q})_T, \qquad\quad\;\;\boldsymbol{q}\in \mathbb P_{k-2}(T;\mathbb T).
\end{align*}
Clearly we have
\begin{align*}
&	\|\boldsymbol{\tau}_h\|\lesssim \|\grad_{\mathcal T_h}\curl\boldsymbol{v}_h\|+\left\|\boldsymbol{n}_F\times \curl\boldsymbol{v}_h-\boldsymbol{\mu}_h\right\|_{1 / 2, h},\\
&	b_h\left(\boldsymbol{\tau}_h ,0; \boldsymbol{v}_h\right) + c_h(\boldsymbol{\tau}_h, \boldsymbol{\mu}_h)=\|\grad_{\mathcal T_h}\curl\boldsymbol{v}_h\|^2+\left\|\boldsymbol{n}_F\times \curl\boldsymbol{v}_h-\boldsymbol{\mu}_h\right\|_{1 / 2, h}^{2}.
\end{align*}
Then
\begin{align}
\label{eq:discreteinfsuph1}
&\quad \|\grad_{\mathcal T_h}\curl\boldsymbol{v}_h\|+\left\|\boldsymbol{n}_F\times \curl\boldsymbol{v}_h-\boldsymbol{\mu}_h\right\|_{1 / 2, h} \\
\notag
&\qquad\qquad\qquad\qquad\qquad\qquad\lesssim \sup _{\boldsymbol{\tau}_h \in \Sigma_{k-1,h}^{-1}} \frac{b_h\left(\boldsymbol{\tau}_h,0; \boldsymbol{v}_h\right) + c_h(\boldsymbol{\tau}_h,\boldsymbol{\mu}_h)}{\left\|\boldsymbol{\tau}_h\right\|},
\end{align}
which together with \eqref{eq:gradcurlhnormequiv} indicates
\begin{align*}
&\quad\; \| (\grad\curl)_h \bs v_h\|+\left\|\boldsymbol{n}_F\times \curl\boldsymbol{v}_h-\boldsymbol{\mu}_h\right\|_{1 / 2, h} \\
&\lesssim \sup _{\boldsymbol{\tau}_h \in \Sigma_h^{-1},\psi_h\in\mathring{\mathbb{V}}_h^{\grad}} \frac{b_h\left(\boldsymbol{\tau}_h ,\psi_h; \boldsymbol{v}_h\right) + c_h(\boldsymbol{\tau}_h,\boldsymbol{\mu}_h)}{\left\|\boldsymbol{\tau}_h\right\|+|\psi_h|_1}.
\end{align*}
Therefore the inf-sup condition \eqref{hybridizationinfsup} holds from \eqref{eq:curldivdiscreteinfsup2}.
\end{proof}


\begin{theorem}
The hybridized mixed finite element method \eqref{hybridishhjtype1}-\eqref{hybridishhjtype2} is well-posed, and the solution $(\boldsymbol{\sigma}_h, \boldsymbol{u}_h, \phi_h) \in {\Sigma}^{\rm tn}_h \times \mathring{\mathbb{V}}^{\curl}_{(k,\ell), h} \times \mathring{\mathbb{V}}^{\grad}_{\ell+1, h}$ coincides with the mixed finite element method \eqref{distribumfem1}-\eqref{distribumfem2}. 
\end{theorem}
\begin{proof}
For $\boldsymbol{\tau}_h\in\Sigma_h^{-1}$ and $\psi_h\in\mathring{\mathbb V}_h^{\grad}$ satisfying 
\begin{equation*}
b_h(\boldsymbol{\tau}_h ,\psi_h; \boldsymbol{v}_h) + c_h(\boldsymbol{\tau}_h,\boldsymbol{\mu}_h)=0,\quad \boldsymbol{v}_h \in \mathring{\mathbb V}_h^{\curl}, \boldsymbol{\mu}_h \in \Lambda_h,
\end{equation*}
we have $\boldsymbol{\tau}_h\in{\Sigma}^{\rm tn}_h$, and $b_h(\boldsymbol{\tau}_h,\psi_h;\boldsymbol{v}_h)=0$ for $\boldsymbol{v}_h \in \mathring{\mathbb V}_h^{\curl}.$
By \eqref{eq:curldivcoercivitydiscrete}, we obtain the discrete coercivity
$\|\boldsymbol{\tau}_h\|^2+|\psi_h|_1^2 \lesssim \|\boldsymbol{\tau}_h\|^2.$
Then we get the well-posedness of the hybridized mixed finite element method \eqref{hybridishhjtype1}-\eqref{hybridishhjtype2} by applying the Babu{\v{s}}ka-Brezzi theory~\cite{BoffiBrezziFortin2013} with the discrete inf-sup condition \eqref{hybridizationinfsup}.

By \eqref{hybridishhjtype2} with $\boldsymbol{v}_{h}=0$, $\boldsymbol{\sigma}_h \in {\Sigma}^{\rm tn}_h$, thus $(\boldsymbol{\sigma}_h, \boldsymbol{u}_h, \phi_h)$ satisfies \eqref{distribumfem1}-\eqref{distribumfem2}. 
\end{proof}

\begin{theorem}
Let $(\boldsymbol{\sigma},0,\boldsymbol{u})$ and $(\boldsymbol{\sigma}_h,0,\boldsymbol{u}_h,\bs \lambda_h)$ be the solution of the mixed formulation~\eqref{quadcurlmixed1}-\eqref{quadcurlmixed2} and the mixed finite element method \eqref{eq:hy} respectively. Assume $\boldsymbol{\sigma}\in H^k(\Omega;\mathbb T)$ and $\boldsymbol{u}, \curl\boldsymbol{u}\in H^{k}(\Omega;\mathbb R^3)$. Then
\begin{equation}\label{hybrierror1}
\left\|\boldsymbol{n}_F\times \curl\boldsymbol{u}_h-\boldsymbol{\lambda}_h\right\|_{-1 / 2, h}\lesssim h^{k}(h|\boldsymbol{\sigma}|_{k}+h|\boldsymbol{u}|_{k}+|\curl\boldsymbol{u}|_{k}).
\end{equation}
\end{theorem}
\begin{proof}
By the proof of the discrete inf-sup condition \eqref{hybridizationinfsup},
\begin{align*}
	&	\left\|\boldsymbol{n}_F\times \curl(I_h^{\curl}\boldsymbol{u}-\boldsymbol{u}_h)-(Q_{\mathcal F_h}^{k-1}(\boldsymbol{n}_F\times \curl\boldsymbol{u})-\boldsymbol{\lambda}_h)\right\|_{1 / 2, h}\\
	\lesssim& \sup _{\boldsymbol{\tau}_h \in\Sigma_{k-1,h}^{-1}} \frac{b_h(\boldsymbol{\tau}_h,0; I_h^{\curl}\boldsymbol{u}-\boldsymbol{u}_h) + c_h(\boldsymbol{\tau}_h,Q_{\mathcal F_h}^{k-1}(\boldsymbol{n}_F\times \curl\boldsymbol{u})-\boldsymbol{\lambda}_h)}{\left\|\boldsymbol{\tau}_h\right\|}.
\end{align*}
Thanks to \eqref{eq:curdivhIhcurl} and \eqref{hybridishhjtype1}, we have
\begin{equation*}
b_h(\boldsymbol{\tau}_h,0; I_h^{\curl}\boldsymbol{u}-\boldsymbol{u}_h) + c_h(\boldsymbol{\tau}_h,Q_{\mathcal F_h}^{k-1}(\boldsymbol{n}_F\times \curl\boldsymbol{u})-\boldsymbol{\lambda}_h)=(\boldsymbol{\sigma}_h-\boldsymbol{\sigma}, \boldsymbol{\tau}_h).
\end{equation*}
Hence
\begin{equation*}
\left\|\boldsymbol{n}_F\times \curl(I_h^{\curl}\boldsymbol{u}-\boldsymbol{u}_h)-(Q_{\mathcal F_h}^{k-1}(\boldsymbol{n}_F\times \curl\boldsymbol{u})-\boldsymbol{\lambda}_h)\right\|_{1 / 2, h}\lesssim \|\boldsymbol{\sigma}-\boldsymbol{\sigma}_h\|.
\end{equation*}
Therefore we can derive estimate \eqref{hybrierror1} from \eqref{eq:errorestimate1} and the error estimate of $I_h^{\curl}$. 
\end{proof}

\subsection{Connection to other methods}
The pair $(\boldsymbol{v}_h, \boldsymbol{\mu}_h)$ can be understood as a weak function. Define $(\grad\curl)_w: \mathring{\mathbb{V}}_h^{\curl} \times\Lambda_h \rightarrow {\Sigma}_h^{-1}$ by
\begin{equation*}
\left((\grad\curl)_w(\boldsymbol{v}_h, \boldsymbol{\mu}_h), \boldsymbol{\tau}_h\right)=-b_h\left(\boldsymbol{\tau}_h ,0; \boldsymbol{v}_h\right) - c_h(\boldsymbol{\tau}_h, \boldsymbol{\mu}_h), \quad \boldsymbol{\tau}_h \in {\Sigma}_{k-1,h}^{-1}.
\end{equation*}
By \eqref{hybridishhjtype1}, we can eliminate $\bs \sigma_h$ element-wise and write $\boldsymbol{\sigma}_h=(\grad\curl)_w(\boldsymbol{u}_h, \boldsymbol{\lambda}_h)$.
Then the hybridized mixed finite element method \eqref{hybridishhjtype1}-\eqref{hybridishhjtype2} can be recast as: find $(\boldsymbol{u}_h, \lambda_h, \phi_h) \in \mathring{\mathbb{V}}_{(k,\ell),h}^{\curl} \times\Lambda_{k-1,h}\times\mathring{\mathbb{V}}_{\ell +1, h}^{\grad}$ such that
\begin{align*}
((\grad\curl)_w(\boldsymbol{u}_h, \boldsymbol{\lambda}_h), (\grad\curl)_w(\boldsymbol{v}_h, \boldsymbol{\mu}_h))- (\boldsymbol{v}_h,\grad\phi_h) &=\langle\boldsymbol{f}, \boldsymbol{v}_h\rangle, \\
(\boldsymbol{u}_h,\grad\psi_h) &=0, 
\end{align*}
for all $(\boldsymbol{v}_h, \mu_h, \psi_h) \in \mathring{\mathbb{V}}_{(k,\ell),h}^{\curl} \times\Lambda_{k-1,h}\times\mathring{\mathbb{V}}_{\ell +1, h}^{\grad}$. That is we obtain a stabilization free weak Galerkin method for solving the quad-curl problem. The stabilization free is due to the inf-sup condition \eqref{eq:discreteinfsuph1}.
Indeed, the inf-sup condition \eqref{eq:discreteinfsuph1} is equivalent to
\begin{equation*}
\|\grad_{\mathcal T_h}\curl\boldsymbol{v}_h\|+\left\|\boldsymbol{n}_F\times \curl\boldsymbol{v}_h-\boldsymbol{\mu}_h\right\|_{1 / 2, h}\lesssim \|(\grad\curl)_w(\boldsymbol{v}_h, \boldsymbol{\mu}_h)\|
\end{equation*}
for $\boldsymbol{v}_h \in \mathring{\mathbb{V}}_{(k,\ell),h}^{\curl}$ and $\boldsymbol{\mu}_h \in \Lambda_{k-1,h}$.
This means $\|(\grad\curl)_w(\boldsymbol{v}_h, \boldsymbol{\mu}_h)\|$ is a norm on space $K_h^c\times \Lambda_{k-1,h}$.

For the hybridized mixed finite element method \eqref{hybridishhjtype1}-\eqref{hybridishhjtype2} of the lowest order $k=1$ and $\ell=0,1$, it is also related to a nonconforming finite element method. 

We first recall the $H(\grad\curl)$ nonconforming finite elements constructed in~\cite{Huang2020,ZhengHuXu2011}. The space of shape functions is 
$\grad\mathbb P_{\ell+1}(T) \oplus (\boldsymbol{x}-\boldsymbol{x}_T) \times \mathbb{P}_{1}\left(T ; \mathbb{R}^{3}\right)$ for $\ell=0,1$. The degrees of freedom are given by
\begin{subequations}
\begin{align}
\label{Hgradcurldof1}	\int_{e} \boldsymbol{v} \cdot \boldsymbol{t}\, q \dd s, & \;\quad q \in \mathbb{P}_{\ell}(e) \text{ on each } e \in \mathcal{E}(T),\\
\label{Hgradcurldof2}	\int_{F}(\curl \boldsymbol{v}) \times \boldsymbol{n}\dd S & \quad  \text{ on each } F \in \mathcal{F}(T).
\end{align}
\end{subequations}
Define the global $H(\grad \curl)$-nonconforming
element space
\begin{align*}
	W_h:=&\{\boldsymbol{v}_h \in {L}^2(\Omega ; \mathbb{R}^3):\,\boldsymbol{v}_h|_T \in \grad\mathbb P_{\ell+1}(T) \oplus (\boldsymbol{x}-\boldsymbol{x}_T) \times \mathbb{P}_{1}\left(T ; \mathbb{R}^{3}\right),\ \forall~T \in \mathcal{T}_h,\\
	&\text{ all the DoFs \eqref{Hgradcurldof1}-\eqref{Hgradcurldof2} are single-valued, and vanish on boundary} \}.
\end{align*}
The interpolation operator $I_h^{\curl}$ can be extended to $W_h$, which is well-defined. We have
$ I_h^{\curl}W_h=\mathring{\mathbb{V}}_h^{\curl}$, and \eqref{eq:curdivhIhcurl} still holds for $\bs v\in W_h$.

An $H(\grad\curl)$-nonconforming finite element method for quad-curl problem \eqref{bigradcurlproblem} is to find $\boldsymbol{w}_h \in W_h$ and $\phi_h \in\mathring{\mathbb{V}}_h^{\grad}$
such that
\begin{subequations}
 \begin{align}
\label{Hgradcurlvariable1}
(\grad_{\mathcal T_h}\curl_{\mathcal T_h}\boldsymbol{w}_h, \grad_{\mathcal T_h}\curl_{\mathcal T_h}\boldsymbol{v}_h)- (I_h^{\curl}\boldsymbol{v}_h,\grad\phi_h) &=\langle\boldsymbol{f}, I_h^{\curl}\boldsymbol{v}_h\rangle, \\
\label{Hgradcurlvariable2}
(I_h^{\curl}\boldsymbol{w}_h,\grad\psi_h) &=0,
\end{align}
\end{subequations}
for all $\boldsymbol{v}_h \in W_h$ and $\psi_h \in\mathring{\mathbb{V}}_h^{\grad}$.
Nonconforming finite element method \eqref{Hgradcurlvariable1}-\eqref{Hgradcurlvariable2} is a modification of those in~\cite{Huang2020,ZhengHuXu2011} by introducing interpolation operator $I_h^{\curl}$. In other words, we identify the complex that accommodates the non-conforming finite elements constructed in~\cite{Huang2020,ZhengHuXu2011} and generalize these elements to arbitrary orders.

\begin{lemma}
For $\boldsymbol{v}_h\in W_h$ satisfying $(I_h^{\curl}\boldsymbol{v}_h,\grad q_h)=0$ for all $q_h\in\mathring{\mathbb{V}}_h^{\grad}$, we have the discrete Poincar\'e inequality
\begin{equation}\label{eq:discretePoincareWh}	
\|\boldsymbol{v}_h\|+\|\curl_{\mathcal T_h}\boldsymbol{v}_h\|\lesssim \|\grad_{\mathcal T_h}\curl_{\mathcal T_h}\boldsymbol{v}_h\|.
\end{equation}
\end{lemma}
\begin{proof}
By (4.12) in~\cite{Huang2020}, it follows
\begin{equation*}
\|\boldsymbol{v}_h-I_h^{\curl}\boldsymbol{v}_h\|\lesssim h\|\curl_{\mathcal T_h}\boldsymbol{v}_h\|.
\end{equation*}
Thanks to the discrete Poincar\'e inequality for space $\mathring{\mathbb{V}}_h^{\curl}$, we have
\begin{equation*}
\|I_h^{\curl}\boldsymbol{v}_h\|\lesssim \|\curl(I_h^{\curl}\boldsymbol{v}_h)\|.
\end{equation*}
Combining the last two inequalities gives
\begin{equation*}
\|\boldsymbol{v}_h\|\leq \|\boldsymbol{v}_h-I_h^{\curl}\boldsymbol{v}_h\|+ \|I_h^{\curl}\boldsymbol{v}_h\|\lesssim h\|\curl_{\mathcal T_h}\boldsymbol{v}_h\|+\|\curl(I_h^{\curl}\boldsymbol{v}_h)\|.
\end{equation*}
By the commutative property of $I_h^{\curl}$, it holds that 
\begin{equation*}
\|\curl(I_h^{\curl}\boldsymbol{v}_h)\|\lesssim \|\grad_{\mathcal T_h}\curl_{\mathcal T_h}\boldsymbol{v}_h\|.
\end{equation*}
Hence
\begin{equation*}
\|\boldsymbol{v}_h\|\lesssim \|\curl_{\mathcal T_h}\boldsymbol{v}_h\|+\|\grad_{\mathcal T_h}\curl_{\mathcal T_h}\boldsymbol{v}_h\|.
\end{equation*}
Finally apply the discrete Poincar\'e inequality for $H^1$-nonconforming linear element to end the proof.
\end{proof}


Using the discrete Poincar\'e inequality \eqref{eq:discretePoincareWh}, we have the well-posedness.
\begin{lemma}
Nonconforming finite element method \eqref{Hgradcurlvariable1}-\eqref{Hgradcurlvariable2} is well-posed.
\end{lemma}

We will show the equivalence between the hybridized mixed finite element method \eqref{hybridishhjtype1}-\eqref{hybridishhjtype2} with $k=1$ and nonconforming finite element method \eqref{Hgradcurlvariable1}-\eqref{Hgradcurlvariable2}.

\begin{theorem}
Let $(\boldsymbol{w}_h, \phi_h) \in W_h\times\mathring{\mathbb{V}}_h^{\grad}$
be the solution of the nonconforming finite element method \eqref{Hgradcurlvariable1}-\eqref{Hgradcurlvariable2}. Then $(\grad_{\mathcal T_h}\curl_{\mathcal T_h}\boldsymbol{w}_h, I_h^{\curl}\boldsymbol{w}_h, \phi_h, Q_{\mathcal{F}_h}(\boldsymbol{n}_F\times\curl\boldsymbol{w}_h)) \in {\Sigma}_{0,h}^{-1} \times \mathring{\mathbb{V}}^{\curl}_h \times \mathring{\mathbb{V}}^{\grad}_h \times \Lambda_h$ is the solution of the hybridized mixed finite element method \eqref{hybridishhjtype1}-\eqref{hybridishhjtype2} with $k = 1$.
\end{theorem}

\begin{proof}
Choose $\boldsymbol{v}_h \in W_h$
such that DoF \eqref{Hgradcurldof1} vanishes, then $I_h^{\curl}\boldsymbol{v}_h = 0$. Applying the integration by parts on the left hand side of \eqref{Hgradcurlvariable1}, we get 
\begin{align*}
0&=\sum_{T\in\mathcal T_h}((\grad\curl\boldsymbol{w}_h)\boldsymbol{n},\curl\boldsymbol{v}_h)_{\partial T}\\
&=\sum_{T\in\mathcal T_h}(\boldsymbol{n}\times(\grad\curl\boldsymbol{w}_h)\boldsymbol{n},\boldsymbol{n}\times\curl\boldsymbol{v}_h)_{\partial T} \\
&=\sum_{F\in\mathring{\mathcal{F}}_h}(\llbracket\boldsymbol{n}\times(\grad_{\mathcal T_h}\curl_{\mathcal T_h}\boldsymbol{w}_h)\boldsymbol{n}\rrbracket, \boldsymbol{n}_F\times\curl_{\mathcal T_h}\boldsymbol{v}_h)_{F}.	
\end{align*}
By the arbitrariness of the DoF \eqref{Hgradcurldof2} for $\boldsymbol{v}_h$, we obtain
$\llbracket\boldsymbol{n}\times(\grad_{\mathcal T_h}\curl_{\mathcal T_h}\boldsymbol{w}_h)\boldsymbol{n}\rrbracket_F= 0$ for all $F \in \mathring{\mathcal{F}}_h$, that is $\grad_{\mathcal T_h}\curl_{\mathcal T_h}\boldsymbol{w}_h \in\Sigma_h^{\rm tn}$.
For all $\boldsymbol{v}_h \in W_h$ and $\boldsymbol{\mu}_h \in \Lambda_h$, we get from
\eqref{eq:curdivhIhcurl}, the integration by parts, and \eqref{Hgradcurlvariable1} that
\begin{align*}
&\quad b_h(\grad_{\mathcal T_h}\curl_{\mathcal T_h}\boldsymbol{w}_h,\phi_h; I_h^{\curl}\boldsymbol{v}_h) + c_h(\grad_{\mathcal T_h}\curl_{\mathcal T_h}\boldsymbol{w}_h, \boldsymbol{\mu}_h) \\
&= b_h(\grad_{\mathcal T_h}\curl_{\mathcal T_h}\boldsymbol{w}_h,\phi_h; I_h^{\curl}\boldsymbol{v}_h) \\
&= b_h(\grad_{\mathcal T_h}\curl_{\mathcal T_h}\boldsymbol{w}_h,0; \boldsymbol{v}_h) +(I_h^{\curl}\boldsymbol{v}_h, \grad\phi_h) \\
&=-(\grad_{\mathcal T_h}\curl_{\mathcal T_h}\boldsymbol{w}_h, \grad_{\mathcal T_h}\curl_{\mathcal T_h}\boldsymbol{v}_h) +(I_h^{\curl}\boldsymbol{v}_h,\grad\phi_h)=-\langle\boldsymbol{f}, I_h^{\curl}\boldsymbol{v}_h\rangle.
\end{align*}
Notice that $ I_h^{\curl}: W_h\to\mathring{\mathbb{V}}_h^{\curl}$ is onto, hence $(\grad_{\mathcal T_h}\curl_{\mathcal T_h}\boldsymbol{w}_h, \phi_h)$ satisfies \eqref{hybridishhjtype2}.

On the side hand,
for all $\boldsymbol{\tau}_{h} \in \Sigma_{0,h}^{-1}$ and $\psi_h \in\mathring{\mathbb{V}}_h^{\grad}$,
apply \eqref{eq:curdivhIhcurl}, \eqref{Hgradcurlvariable2} and the integration by parts to get
\begin{align*}
&\quad (\grad_{\mathcal T_h}\curl_{\mathcal T_h}\boldsymbol{w}_h, \boldsymbol{\tau}_{h})+b_h\left(\boldsymbol{\tau}_{h},\psi_h; I_h^{\curl}\boldsymbol{w}_h\right) + c_h(\boldsymbol{\tau}_{h},  Q_{\mathcal{F}_h}(\boldsymbol{n}_F\times\curl\boldsymbol{w}_h)) \\
&=(\grad_{\mathcal T_h}\curl_{\mathcal T_h}\boldsymbol{w}_h, \boldsymbol{\tau}_{h})+b_h\left(\boldsymbol{\tau}_{h},0; \boldsymbol{w}_h\right) + c_h(\boldsymbol{\tau}_{h},  \boldsymbol{n}_F\times\curl\boldsymbol{w}_h) \\
&\quad+(I_h^{\curl}\boldsymbol{w}_h,\grad\psi_h) \\
&=(\grad_{\mathcal T_h}\curl_{\mathcal T_h}\boldsymbol{w}_h, \boldsymbol{\tau}_{h})+b_h\left(\boldsymbol{\tau}_{h},0; \boldsymbol{w}_h\right) -\sum_{T \in \mathcal{T}_{h}}\left(\boldsymbol{n}\times \boldsymbol{\tau}_h\boldsymbol{n}, \boldsymbol{n}\times\curl\boldsymbol{w}_h\right)_{\partial T}=0.
\end{align*}
That is $(\grad_{\mathcal T_h}\curl_{\mathcal T_h}\boldsymbol{w}_h, I_h^{\curl}\boldsymbol{w}_h, Q_{\mathcal{F}_h}(\boldsymbol{n}_F\times\curl\boldsymbol{w}_h))$ satisfies \eqref{hybridishhjtype1}.
\end{proof}

\bibliographystyle{abbrv}
\bibliography{./refs}
\end{document}